\documentclass[11pt,a4paper]{amsart}

\usepackage{amsmath,amssymb,enumerate,amsthm,graphicx}
\usepackage[initials]{amsrefs}

\numberwithin{equation}{section}

\DeclareSymbolFont{symbolsC}{U}{txsyc}{m}{n}
\DeclareMathSymbol{\diamonddot}{\mathord}{symbolsC}{144} 
\DeclareMathSymbol{\diamondopen}{\mathord}{symbolsC}{94} 

\renewcommand{\Diamond}{\mathop\diamondopen} 
\newcommand{\DIAMOND}{\mathop\diamonddot} 
\newcommand{\BOX}{\mathop\boxdot}  
\renewcommand{\Box}{\mathop\square} 
 
\newtheorem{theorem}{Theorem}[section]
\newtheorem{lemma}[theorem]{Lemma}
\newtheorem{corollary}[theorem]{Corollary}
\newtheorem{proposition}[theorem]{Proposition}

\theoremstyle{definition}
\newtheorem{remark}[theorem]{Remark}
\newtheorem{example}[theorem]{Example}

\newcommand{\TABS}{%
\qquad\qquad\qquad\qquad\qquad\qquad\qquad\qquad\qquad\= \kill}
\newcommand{\FUSE}{\ensuremath{\otimes}}%
\newcommand{\PLUS}{\ensuremath{\oplus}}

\begin{document}

\title[Characterizing intermediate tense logics]{Characterizing intermediate tense logics in terms of Galois connections}

\author[W.~Dzik]{Wojciech Dzik}
\address{W.~Dzik, Institute of  Mathematics, University of Silesia, ul.~Bankowa 12,~40-007 Katowice, Poland}
\email{wojciech.dzik@us.edu.pl}

\author[J.~J{\"a}rvinen]{Jouni J{\"a}rvinen}
\address{J.~J{\"a}rvinen, Sirkankuja 1, 20810~Turku, Finland}
\email{Jouni.Kalervo.Jarvinen@gmail.com}

\author[M.~Kondo]{Michiro Kondo}
\address{M.~Kondo, School of Information Environment, Tokyo Denki University, Inzai, 270-1382, Japan}
\email{mkondo@mail.dendai.ac.jp}

\begin{abstract}
We propose a uniform way of defining for every logic {\sf L} intermediate between intuitionistic and classical logics, 
the corresponding intermediate tense logic {\sf LK$_t$}. This is done by building the fusion of two copies of 
intermediate logic with a Galois connection {\sf LGC}, and then interlinking their operators by two Fischer Servi axioms.  
The resulting system is called {\sf L2GC+FS}. In the cases of intuitionistic logic {\sf Int} and classical logic {\sf Cl}, 
it is noted that {\sf Int2GC+FS} is syntactically equivalent to intuitionistic tense logic {\sf IK$_t$} by W.~B.~Ewald 
and {\sf Cl2GC+FS} equals classical tense logic {\sf K$_t$}. This justifies calling {\sf L2GC+FS} the {\sf L}-tense logic 
{\sf LK$_t$} for any intermediate logic {\sf L}. We define H2GC+FS-algebras as expansions of HK1-algebras, 
introduced by E.~Or{\l}owska and I.~Rewitzky. For each intermediate logic {\sf L}, we show algebraic completeness of 
{\sf L2GC+FS} and its conservativeness over {\sf L}. We  prove relational completeness of {\sf Int2GC+FS} with respect to the models 
defined on {\sf IK}-frames introduced by G.~Fischer Servi. We also prove a representation theorem stating that every H2GC+FS-algebra 
can be embedded into the complex algebra of its canonical {\sf IK}-frame.
\end{abstract}

\maketitle

\section{Introduction} \label{Section:Intro}

In this paper, we consider the following  method of introducing unary operators to intuitionistic propositional logic: 
\begin{enumerate}[\rm (A)]
\item Building the fusion {\sf IntGC}\FUSE{\sf IntGC} of two copies of intuitionistic logic with a Galois connection {\sf IntGC}, 
the first one with a Galois connection $({\Diamond},{\BOX})$ and the second one with $({\DIAMOND},{\Box})$, and adding Fischer Servi axioms 
to connect  $({\Diamond},{\Box})$ and $({\DIAMOND},{\BOX})$. 
\end{enumerate}
Another method of introducing unary operators leading to intuitionistic tense logic was investigated by J.~M.~Davoren \cite{Davoren}: 
\begin{enumerate}[\rm (B)]
\item Building the fusion {\sf IK}\FUSE{\sf IK} of two copies of intuitionistic modal logic {\sf IK}, the first one with modalities $({\Diamond},{\Box})$ 
and the second one with $({\DIAMOND},{\BOX})$, and adding Brouwerian axioms to connect $({\Diamond},{\BOX})$ and $({\DIAMOND},{\Box})$.
\end{enumerate}
These two methods are shown here to be equivalent and the result is called {\sf Int2GC+FS}, according to (A). This name should be understood
as ``intuitionistic logic with two Galois connections combined using Fischer Servi axioms''.

Note that for combinations of modal logics, we follow the notation of \cite{Davoren}. If $\mathcal{L}_1$ and $\mathcal{L}_2$ are axiomatically presented modal 
logics in languages $\Lambda_1$ and $\Lambda_2$, respectively, then the fusion $\mathcal{L}_1$\FUSE$\mathcal{L}_2$ is the smallest multi-modal logic in the 
language $\Lambda_1$\FUSE$\Lambda_2$ containing $\mathcal{L}_1$ and $\mathcal{L}_2$, and closed under all the inference rules of  $\mathcal{L}_1$ and $\mathcal{L}_2$, 
where $\Lambda_1$\FUSE$\Lambda_2$  denotes the smallest common extension of the languages  $\Lambda_1$ and $\Lambda_2$. If $\mathcal{L}$ is a logic in language 
$\Lambda$, and $\Gamma$ is a finite set of schemes in $\Lambda$, then the extension $\mathcal{L} \PLUS \Gamma$ is the smallest logic in $\Lambda$ extending
$\mathcal{L}$, containing the schemes in $\Gamma$ as additional axioms, and closed under the rules of $\mathcal{L}$. 

If $\Diamond$, $\Box$  and  $\DIAMOND$, $\BOX$  are identified with tense operators $F$, $G$ (future) and $P$, $H$ (past), respectively,
the system {\sf Int2GC+FS} is equivalent to the known system {\sf IK$_t$}, called \emph{intuitionistic tense logic},
introduced by W.~B.~Ewald \cite{Ewald86}.\footnote{This equivalence was proved already in \cite{DzJaKo12A}.}
The logic {\sf IK$_t$} is generally taken as the intuitionistic counterpart of the \emph{classical tense logic} {\sf K$_t$} 
(see \cite{Davoren, simpson1994proof}, for instance) and we will 
neither discuss this fact here nor consider the philosophical issues raised by {\sf IK$_t$} (for instance, its constructivity).  
We would also like to emphasize that this is not a matter of providing another list of axioms for {\sf IK$_t$} 
that is much shorter than the Ewald's list of axioms. 
Note that the logic {\sf K$_t$} is often in the literature called the \emph{minimal tense logic}. Since we consider only the 
minimal tense (classical,  intuitionistic, intermediate) logics, we will omit the word ``minimal'' in the rest of the paper.
Methods (A) and (B) are visualized in Figure~\ref{Fig:figure1}.
\begin{figure}[h] \label{Fig:figure1} \centering
\includegraphics[width=100mm]{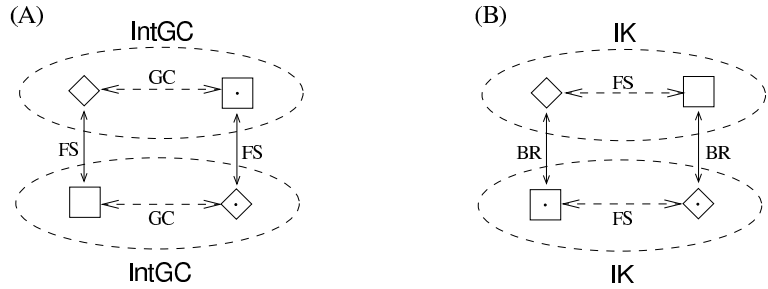}
\caption{\small Two methods of building {\sf IK$_t$}. Here FS stands for the Fischer Servi axioms, GC for Galois connections, 
and BR for the Browerian axioms.}
\end{figure}

The above equivalences also hold if one changes the basic logic from intuitionistic to classical, in which case one gets 
classical tense logic {\sf K$_t$}. Adopting approach (A) from intuitionistic and classical logics to any intermediate 
logic {\sf L}, we present a method to obtain the corresponding logic {\sf L2GC+FS}. This is done simply by adding to 
{\sf L} two Galois connections (by means of the appropriate rules), or by building the fusion
{\sf LGC}\FUSE{\sf LGC} of two copies of intermediate logic  with a Galois connection {\sf LGC}, and then interlinking their operators 
by two Fischer Servi axioms.  We prove algebraic completeness for {\sf L2GC+FS} and show that it is conservative over {\sf L}. 
We also give facts justifying why {\sf L2GC+FS} can be considered as an \emph{intermediate {\sf L}-tense logic} {\sf LK$_t$}.

There are several advantages of approach (A) over (B). The {\sf L}-tense logic {\sf LK$_t$} (or equivalently {\sf L2GC+FS}) 
can be uniformly built for every intermediate logic {\sf L}, without entering the problem of what is the modal version {\sf LK} of {\sf L},  
since the ``modal part'' is  provided solely by the Galois connections, and the Fischer Servi axioms make a duality-like connection between the
operators. For a given intermediate logic {\sf L}, it is often not clear what is its modal analogue {\sf LK} (between {\sf IK} and {\sf K}). 
For instance, it took several years to find out, what is G{\"o}del modal logic. In \cite{Caicedo10}, strong completeness of the $\Box$-version 
and the $\Diamond$-version of G{\"o}del modal logic were proved. Recent studies \cite{Caicedo12, Caicedo13} show that there are several G{\"o}del 
modal logics of two modalities which are defined by a Kripke frame semantics. In particular, G{\"o}del modal logics  are different for ``crisp'' 
frames and for ``fuzzy'' frames. Moreover, approach (A) allows a uniform treatment of algebraic semantics.

Galois connections play a central role both in (A) and (B) -- and in the whole paper, hence we recall some well-known properties of order-preserving Galois 
connections used here. They can be found in \cite{ErKoMeSt93}, for instance. Let $\varphi \colon P \to Q$ and $\psi \colon Q \to P$ be maps between
ordered sets $P$ and $Q$. The pair $(\varphi,\psi)$ is a \emph{Galois connection} between $P$ and $Q$, if for all $p \in P$ and $q \in Q$,
\[
 \varphi(p) \leq q  \iff p \le \psi(q).
\]
An equivalent characterisation states that a pair $(\varphi,\psi)$ forms a Galois connection between $P$ and $Q$ if and only if
\begin{align}
&\text{$p \leq \psi (\varphi (p))$ for all $p \in P$ and $\varphi ( \psi(q)) \leq q$ for all $q \in Q$;} \label{EQ:GC1}\\
&\text{the maps $\varphi$ and $\psi$ are order-preserving.} \label{EQ:GC2}
\end{align}

It is well known that  Galois connections can be created by any relational frame $(U,R)$ by reversing the relation $R$.  
The operators $\Diamond$ and $\Box$ defined for all $X \subseteq U$ by
$\Diamond X = \{ x \in U \mid (\exists y \in U)\, x \, R \, y \ \& \ y \in X \}$ and 
$\Box X = \{x \in U \mid (\forall y \in U)\, x\, R \, y \Rightarrow y \in X \}$
are both part of a Galois connection.
The Galois connections in question on the powerset lattice $\wp(U)$ are then
$({\Diamond},{\BOX})$ and $({\DIAMOND},{\Box})$, where the operators
$X \mapsto \DIAMOND X$ and $X \mapsto \BOX X$ are defined by inverting
the relation $R$. However, the idea of extending propositional calculus with a Galois 
connection as modalities appears to be rather new, and mainly motivated by applications 
in computer science. There is a growing interest in the study of Galois connections 
as modalities, as can be seen in the recent surveys by M.~Menni and C.~Smith
\cite{Menni2013} and Garc\'ia-Pardo \textit{et al.}\@ \cite{GarciaPardo2013}.
The study of Galois connections can be traced back to the initial works of O.~Ore \cite{Ore44} and  B.~J{\'o}nsson and A.~Tarski \cite{JoTa51}. 
More recent studies of Galois connections as modal operators in complete lattices can be found, for instance, 
in  \cite{vonKarger95}, where B.~{von~Karger} developed several temporal logics 
from the theory of complete lattices, Galois connections, and fixed points, and in
\cite{JaKoKo06}, where Galois connections, conjugate functions, and their 
fixed points are considered in complete Boolean lattices.

In ``syntactical side'', Galois connections can be subsumed into a logic only either
by including Galois connection rules (see page~\pageref{Rule:GC}) or by introducing
Browerian axioms (see page~\pageref{Eq:Brower}). However, in ``semantical side'', the
situation is different in the sense that, for instance, for a complete lattice $(L,\leq)$, a mapping 
$f \colon L \to L$ is known to be a part of a Galois connection if and only if $f$ is a complete join-morphism, that is, 
$f (\bigvee S) = \bigvee f(S)$ for all $S \subseteq L$. In such a case, the ``other part'' $g$ is defined by 
$g(a) := \max \{ a \in L \mid f(a) \leq b\}$. This means, for example, that in a finite lattice $(L,\vee,\wedge)$, 
every additive and normal map $L \to L$ induces a Galois connection. In relational settings, Galois
connections are essentially related to inverting a relation; if a possibility-like
operator is defined in terms of a relation (or a composition of relations) by an ``exists''-condition, 
then its adjoint operator is defined simply by a ``for all''-condition and the inverse the original relation 
(or the inverse of the composition of relations). Similar kind of situation can be observed  for ``categorical functors'', 
and a functor is known to have a left adjoint if and only if it is continuous and a certain ``smallness condition'' is satisfied.
Note that every partially ordered set can be viewed as a category in a natural way: there is a unique morphism from $x$ to $y$ if and only 
if $x \leq y$. Thus, an order-preserving Galois connection is a pair of adjoint functors between two categories that arise from 
partially ordered sets.

In the literature can be found several papers that consider modalities as adjoint pairs. Topos-theoretic approaches to modality are presented 
by G.~E.~Reyes and H.~Zolfaghari in \cite{ReyZolfa91}, with adjoint pair $({\Diamond}, {\Box})$, 
and {\sf S4}-like axioms satisfied by $\Box$ and $\Diamond$ separately.  In \cite{ReyZawad91}, G.~E.~Reyes and M.~W.~Zawadowski 
developed this theory further in the context of locales, giving  axiomatisation, completeness and decidability of modal 
logics arising in this context. More recently, M.~Sadrzadeh and R.~Dyckhoff studied in \cite{SadDyc10} positive logic whose 
nonmodal part has conjunction and disjunction as connectives, and whose modalities come in adjoint pairs.

In categorical models, propositions are interpreted as the objects of a category and proofs as morphisms. For instance, 
P.~N.~Benton considers in \cite{Benton95} so-called {\sf LNL}-models, which are categorical models for intuitionistic linear logic as defined by Girard.
Benton studies also rules for {\sf LNL} which are similar to our Galois connection rules.
In \cite{BiermanPaiva99}, G.M.~Bierman and V.~de~Paiva consider an intuitionistic variant {\sf IS4} of the modal logic {\sf S4} and its models 
in the framework of category theory.  Alechina \emph{et al.}\@  study in \cite{alechina} two systems of constructive modal logic which are 
computationally motivated. These logics are ``Constructive S4'' and ``Propositional Lax Logic''. They
prove duality results which show how to relate Kripke models to algebraic models, and these in turn to the appropriate categorical models.
Our work is based on algebraic and Kripke semantics, and since we consider minimal intermediate tense logics, we do not assume additional modal axioms. 
Hence, we do not follow the categorical proposal of modelling constructive {\sf S4}-modalities that uses the additional axioms {\sf T} and {\sf 4}.
The difference between systems applying constructive {\sf S4}-modalities and ours is similar to the difference between classical tense logic 
and classical {\sf S4}.

This paper continues our study of Galois connections in intuitionistic logic. 
In \cite{DzJaKo10}, we introduced intuitionistic propositional logic
with a Galois connection $({\Diamond},{\BOX})$, called {\sf IntGC}. We showed that
${\Diamond}$ and ${\BOX}$ are \emph{modal operators} in the sense that ${\Diamond}$ distributes over $\vee$ 
(that is, is \emph{additive}) and preserves $\perp$ (that is, is \emph{normal})
and ${\BOX}$ distributes over $\wedge$ (i.e., is \emph{multiplicative}) and preserves
$\top$ (i.e., is \emph{co-normal}). We gave both algebraic and relational
semantics, and showed that {\sf IntGC} is complete with respect to  both of
these semantics. We noted that {\sf IntGC} is conservative over intuitionistic logic and 
that Glivenko's Theorem does not hold between propositional logic with a Galois connection
\cite{JaKoKo08} and {\sf IntGC}.
In addition, in \cite{DzJaKo12} we proved that {\sf IntGC} has the finite model property,
which enabled us to state that a formula of {\sf IntGC} is provable if and only if
it is valid in any finite distributive lattice with an additive and normal
operator, or equivalently, the formula is valid in any finite distributive lattice with a 
multiplicative and co-normal operator. With respect to relational semantics,
this is equivalent to the validity in any finite relational models for {\sf IntGC}.
We also presented how {\sf IntGC} is motivated by generalised fuzzy sets. 
In \cite{DzJaKo14}, we gave representations of expansions of bounded distributive lattices equipped 
with a Galois connection.
We studied in \cite{DzJaKo12A} two Galois connections in intuitionistic logic and then with Fischer Servi axioms added,  
their algebraic and relational semantics. We announced there some results that are presented here. 
In a similar way, in this work we extend {\sf IntGC} with a Galois connection $(\Diamond,\BOX)$ by adding another Galois connection 
pair $(\DIAMOND,\Box)$. Just adding another Galois connection does not change much, we have
{\sf IntGC} ``doubled'', called here {\sf Int2GC}. Note that  {\sf Int2GC} is the same as the
fusion {\sf IntGC}\FUSE{\sf IntGC}.
One of the motivating questions of this paper is: What axioms connecting two independent Galois connections 
$(\Diamond,\BOX)$ and $(\DIAMOND,\Box)$ should be added to obtain intuitionistic (or intermediate) tense logic?

In classical logic, the operators $\Diamond$ and $\Box$ may be defined as a shorthand of each 
other by using the following \emph{De Morgan definitions}:
\begin{equation}\label{Eq:DeMorganDef}
 \Box A := \neg \Diamond \neg A \quad \text{ and } \quad \Diamond A := \neg \Box \neg A.
\end{equation}
Classical tense logic {\sf K$_t$} can be obtained by adding to classical logic two Galois connections 
$(\Diamond, \BOX)$ and $(\DIAMOND, \Box)$, and then connecting them by the following \emph{De~Morgan axioms}:
\begin{equation}\label{Eq:DeMorganDual1}
\Box A \leftrightarrow \neg \Diamond \neg A \quad \text{ and } \quad \BOX A \leftrightarrow \neg \DIAMOND \neg A,
\end{equation}
or 
\begin{equation}\label{Eq:DeMorganDual2}
\Diamond A \leftrightarrow  \neg \Box \neg A \quad \text{ and } \quad \DIAMOND A \leftrightarrow  \neg \BOX \neg A 
\end{equation}
Note that in the case of classical logic, the formulas in \eqref{Eq:DeMorganDual1} are equivalent to the ones in 
\eqref{Eq:DeMorganDual2}. In a more concise way, {\sf K$_t$} may be determined by adding to classical 
logic one Galois connection $(\Diamond, \BOX)$, and then defining the second one $(\DIAMOND, \Box)$ in terms of 
\eqref{Eq:DeMorganDef}, that is, by setting $\DIAMOND A := \neg \BOX \neg A$ and $\Box A := \neg \Diamond \neg A$. 
This approach is present in  \cite[Proposition~8.5(iii)]{Venema2007} and also, in another, independent way, 
in \cite{JaKoKo08}.

However, if one changes the base logic from 
classical to intuitionistic, or algebraically from Boolean to Heyting algebras, these kinds of 
ways cannot be used, because they lead to serious faults and fallacies. In particular, having a Galois connection 
$({\Diamond},{\BOX})$, if one defines the operators  $\Box$ and $\DIAMOND$ by using \eqref{Eq:DeMorganDef} 
with intuitionistic negation, the resulting pair $({\DIAMOND},{\Box})$ does not form a Galois connection; 
see \cite{DzJaKo10}. In another similar approach \cite{chajda11} (without using the term Galois connection, but providing 
the equivalent axiomatisation), the assertions \eqref{Eq:DeMorganDef} are used to define ``possibility-like'' 
tense operators $F$, $P$, over intuitionistic logic, from ``necessity-like'' tense operators $G$, $H$. 
It is claimed in  \cite{chajda11} that the resulting logic is intuitionistic  tense logic and 
that the ``possibility-like'' tense operators $F$, $P$ are ``existential quantifiers'' (see \cite[Remark~8]{chajda11}) 
meaning, in particular, that $F$, $P$ preserve disjunctions (that is, lattice-joins). Showing that this is not true is the topic 
of \cite{figallo2012remarks}. Note also that in Example~\ref{Exa:Motivation}(c) we show that in $\mbox{\sf Int2GC+FS} = \mbox{\sf IK$_t$}$, 
formulas \eqref{Eq:DeMorganDual1} and \eqref{Eq:DeMorganDual2} are not provable.

By the above, it is clear that De~Morgan axioms \eqref{Eq:DeMorganDual1} and \eqref{Eq:DeMorganDual2} are not appropriate 
for connecting modalities over intuitionistic logic due to the properties of intuitionistic negation. Our answer to 
the above question on intuitionistic logic level is to link the operators $\Diamond$, $\Box$ and $\DIAMOND$, $\BOX$,
respectively, by using the ``connecting axioms''
\[
 \Diamond( A \to\ B) \to (\Box A \to \Diamond B) \text{ \ and \ } 
(\Diamond A \to \Box B) \to \Box(A \to B) 
\]
introduced by G.~Fischer Servi in \cite{FishServ84}. Note that in these axioms, negation is not involved. To define
{\sf Int2GC+FS}, the two Galois connections $(\Diamond,\BOX)$ and $(\DIAMOND,\Box)$ of  {\sf Int2GC} are interlinked with the axioms: 
\begin{tabbing}\TABS\label{Eq:FS}%
({FS}1) $\Diamond(A \to B) \to (\Box A \to \Diamond B)$ \>({FS}2) $\DIAMOND(A \to B) \to (\BOX A \to \DIAMOND B)$ \\
({FS}3) $(\Diamond A \to \Box B) \to \Box(A \to B)$\>({FS}4) $(\DIAMOND A \to \BOX B) \to \BOX(A \to B)$ 
\end{tabbing}
We will show that in {\sf Int2GC}, axioms (FS1) and (FS4) are equivalent, and the same holds for (FS2) and (FS3), meaning that 
we have some equivalent combinations of axioms to define {\sf Int2GC+FS}, and thus also {\sf IK$_t$}.

Another way of connecting two independent Galois connections, if one moves from Boolean to distributive lattices, is based on J.~M.~Dunn's axioms. 
These axioms connect modalities in positive modal logic.   
In \cite{Dunn94}, Dunn studied distributive lattices with two modal operators $\Box$ and $\Diamond$ and introduced conditions 
\begin{equation} \label{Eq:Dunn}
 \Diamond x \wedge \Box y \leq \Diamond( x \wedge y) \qquad \text{and} \qquad
\Box(x \vee y) \leq \Box x \vee \Diamond y
\end{equation}
for the interactions between $\Box$ and $\Diamond$. We use only the first of them, the second is false in {\sf IK$_t$}. 
In fact, in Heyting algebras with two Galois connections, the conditions of \eqref{Eq:Dunn}  are independent of each other.
One obtains a logic equivalent to {\sf IK$_t$} by adding to {\sf Int2GC} axioms corresponding to the first condition of 
\eqref{Eq:Dunn} applied to the pairs $({\Box},{\Diamond})$  and $({\BOX},{\DIAMOND})$.%
\footnote{Added in proof: after sending the first version of this paper to the editors in 2012 we learned that a similar 
result applying Dunn's axiom was presented in \cite{Menni2013}, appearing while our paper was in reviewing process, see also \cite{DzJaKo12A}.}  
The axioms are ``positive'' -- negation is not present in distributive lattices. 
One may say that the role of linking two Galois connections played by De~Morgan axioms in classical logic is taken by Fischer Servi 
axioms or by (positive) Dunn's axioms, in intuitionistic logic and, more general, in intermediate logics.
 
The next motivation of the paper is to show completeness of the logic for both algebraic and relational semantics, and to find a
representation theorem for Heyting algebras with Galois connections  via relational intuitionistic-modal frames. 
We consider H2GC+FS-algebras, which are algebras $(H,\vee,\wedge,\to,0,1,\Diamond,\Box,\DIAMOND,\BOX)$ such that
$(H,\vee,\wedge,\to,0,1,\Diamond,\BOX)$ and $(H,\vee,\wedge,\to,0,1,\DIAMOND,\Box)$
are HGC-algebras modelling {\sf IntGC} \cite{DzJaKo10},  and $(H,\vee,\wedge,\to,0,1,\Diamond,\Box)$ and
$(H,\vee,\wedge,\to,0,1,\DIAMOND,\BOX)$ are so-called HK1-algebras introduced
by E.~Or{\l}owska and I.~Rewitzky in \cite{OrlRew07}. We note that {\sf Int2GC+FS} is complete with respect to
H2GC+FS-algebras, and we generalise this result to completeness of the logic {\sf L2GC+FS}
for any intermediate logic {\sf L}, with respect to  {\sf L}-Heyting algebras extended with two Galois connection pairs interlinked
with Fischer Servi axioms. We also note that calculating using Heyting algebras with operators
is much easier than calculating with categories, and calculating with algebras can be
easily used in showing some of the non-theorems, for instance, that all Dunn's axioms~\eqref{Eq:Dunn}
are not true in {\sf IK$_t$}.

We recall {\sf IK}-frames from \cite{FishServ84} and show that {\sf Int2GC+FS} is complete with respect 
to the models defined on {\sf IK}-frames. In addition, we prove a representation theorem for H2GC+FS-algebras: 
every H2GC+FS-algebra can be embedded into the complex algebra of its canonical {\sf IK}-frame. 
This is a non-classical generalisation of B.~J{\'o}nsson and A.~Tarski  \cite{JoTa51} representation of Boolean algebras with operators.
Note that ``complex algebra'' is a commonly used name for the standard construction of an algebra of a certain type from a given frame, developed in \cite{JoTa51}.
Contrary to the case of algebraic semantics, relational semantics adequate for {\sf L2GC+FS} does not necessarily exist for every intermediate 
logic {\sf L}, because {\sf L} may itself be Kripke-incomplete. Hence, relational completeness and the representation theorem for 
{H$_{\sf L}$2GC+FS}-algebras for other intermediate logics {\sf L} are left for a separate study.

This paper is structured as follows. In Section~\ref{Sec:Axiomatisations}, we recall the
logic {\sf IntGC} introduced by the authors in \cite{DzJaKo10}. We show that in the
fusion of two {\sf IntGC} logics with two independent Galois connection pairs $({\Diamond},{\BOX})$
and $({\DIAMOND}, {\Box})$, axioms (FS1) and (FS4) are equivalent, and so are (FS2) and (FS3). 
Logic {\sf Int2GC+FS} is then defined as a fusion of two {\sf IntGC}s plus two axioms (FS1) and (FS2) added.
We note that {\sf Int2GC+FS} can be regarded as an intuitionistic bi-modal logic, 
and the pairs $\Diamond$, $\Box$ and  $\DIAMOND$, $\BOX$ are intuitionistic modal connectives 
in the sense of Fischer Servi. In fact, {\sf Int2GC+FS} extends the fusion {\sf IK}\FUSE{\sf IK} by
the Browerian axioms, and this gives us the procedure (B). We also prove that {\sf Int2GC+FS}
is syntactically equivalent to intuitionistic tense logic {\sf IK$_t$}. Section~\ref{Sec:Algebras} is devoted to H2GC+FS-algebras. In this section, also 
fuzzy modal operators on complete Heyting algebras are considered as another motivation.
Section~\ref{Sec:RelationalSematics} contains a relational completeness results showing that {\sf Int2GC+FS} 
is complete with respect to the models defined on {\sf IK}-frames.
We also give a representation theorem stating that any H2GC+FS-algebra can be embedded into the complex 
algebra of its canonical {\sf IK}-frame.
The paper ends with some concluding remarks.

\section{Intuitionistic logic with two Galois connections and Fischer Servi axioms}
\label{Sec:Axiomatisations}

We begin recalling the intuitionistic propositional logic with a Galois connection ({\sf IntGC})
defined by the authors in \cite{DzJaKo10}. The language of {\sf IntGC} is constructed from an enumerable
infinite set of propositional variables $\mathit{Var}$, the connectives $\neg$, $\vee$, $\wedge$, $\to$,
and the unary operators $\Diamond$ and $\BOX$. The constant \emph{true} is defined by setting
$\top := p \to p$ for some fixed propositional variable $p \in \mathit{Var}$, and the constant 
\emph{false} is defined by $\bot := \neg \top$. We also set $A \leftrightarrow B := (A \to B) \wedge (B \to A)$.
The logic {\sf IntGC} is the smallest logic that contains intuitionistic 
propositional logic {\sf Int} and is closed under modus ponens (MP), and rules (GC\,${\BOX}{\Diamond}$) and (GC\,${\Diamond}{\BOX}$):
\begin{tabbing}
\TABS\label{Rule:GC}%
(GC\,${\BOX}{\Diamond}$) \ $\displaystyle \frac{A \to \BOX B}{\Diamond A \to B}$ 
\>(GC\,${\Diamond}{\BOX}$) \ $\displaystyle \frac{\Diamond A \to B}{A \to \BOX B}$ 
\end{tabbing} \medskip
It is known that the following rules are admissible in {\sf IntGC}:

\begin{tabbing}
\TABS
(RN$\BOX$) $\displaystyle \frac{A}{\BOX A}$ \\[4mm]
(RM$\BOX$) $\displaystyle \frac{A \to B}{\BOX A \to \BOX B}$  \>(RM$\Diamond$) $\displaystyle\frac{A \to B}{\Diamond A \to \Diamond B}$ \\
\end{tabbing}
In addition, the following formulas are provable:
\begin{enumerate}[\rm (i)]
\item $A \to \BOX \Diamond A$ \ and \  $\Diamond \BOX A \to A$;

\item $\Diamond A \leftrightarrow \Diamond \BOX \Diamond A$ \ and \ $\BOX A \leftrightarrow \BOX \Diamond \BOX A$;

\item $\BOX \top$  \ and \  $\neg \Diamond \bot$;

\item $\BOX (A \wedge B)  \leftrightarrow  \BOX A \wedge \BOX B$ \ and \
$\Diamond (A \vee B)  \leftrightarrow \Diamond A \vee \Diamond B$;

\item $\BOX (A \to B) \to (\BOX A \to \BOX B)$.
\end{enumerate}

Next we define {\sf Int2GC} by adding another independent Galois connection pair to {\sf IntGC}. The language of the 
logic {\sf Int2GC} is thus the one of {\sf IntGC} extended by two unary connectives 
${\DIAMOND}$ and ${\Box}$, and the logic {\sf Int2GC} is the smallest logic extending {\sf IntGC} by rules
(GC\,${\Box}{\DIAMOND}$) and (GC\,${\DIAMOND}{\Box}$):\medskip
\begin{tabbing}
\TABS
(GC\,${\Box}{\DIAMOND}$) \ $\displaystyle \frac{A \to \Box B}{\DIAMOND A \to B}$ 
\>(GC\,${\DIAMOND}{\Box}$) \ $\displaystyle \frac{\DIAMOND A \to B}{A \to \Box B}$
\end{tabbing}\medskip
Obviously, in {\sf Int2GC} also the rules:
\begin{tabbing}
\TABS
(RN$\Box$) $\displaystyle \frac{A}{\Box A}$ \\[4mm]
(RM$\Box$) $\displaystyle \frac{A \to B}{\Box A \to \Box B}$  \>(RM$\DIAMOND$) $\displaystyle\frac{A \to B}{\DIAMOND A \to \DIAMOND B}$ \\
\end{tabbing}
are admissible, and the following formulas are provable:
\begin{enumerate}[(i)]

\item $A \to \Box \DIAMOND A$ \ and \  $\DIAMOND \Box A \to A$;

\item $\DIAMOND A \leftrightarrow \DIAMOND \Box \DIAMOND A$ \ and \ $\Box A \leftrightarrow \Box \DIAMOND \Box A$;

\item $\Box \top$  \ and \  $\neg \DIAMOND \bot$;

\item $\Box (A \wedge B)  \leftrightarrow  \Box A \wedge \Box B$ \ and \
$\DIAMOND (A \vee B)  \leftrightarrow \DIAMOND A \vee \DIAMOND B$;

\item $\Box (A \to B) \to (\Box A \to \Box B)$.
\end{enumerate}
In fact, {\sf Int2GC} is just the fusion {\sf IntGC}\FUSE{\sf IntGC} of two separate {\sf IntGC}s having the
Galois connections $({\Diamond},{\BOX})$ and $({\DIAMOND},{\Box})$, respectively.

\medskip

\emph{Intuitionistic modal logic}\/ {\sf IK} was introduced by
G.~Fischer Servi in \cite{FishServ84}. The logic {\sf IK} is obtained
by adding two modal connectives $\Diamond$ and $\Box$ to intuitionistic
logic satisfying the following axioms:
\begin{enumerate}[\rm ({IK}1)]\label{Axioms:Fischer}
\item $\Diamond (A \vee B) \to \Diamond A \vee \Diamond B$
\item $\Box A \wedge \Box B \to \Box (A \wedge B)$
\item $\neg \Diamond \bot$
\item $\Diamond( A \to\ B) \to (\Box A \to \Diamond B)$
\item $(\Diamond A \to \Box B) \to \Box(A \to B)$ 
\end{enumerate}
In addition, the \emph{monotonicity rules} for both $\Diamond$ and $\Box$ are admissible:
\begin{tabbing}
\TABS
(RM$\Diamond$) $\displaystyle \frac{A \to B}{\Diamond A \to \Diamond B}$  
\> (RM$\Box$) $\displaystyle\frac{A \to B}{\Box A \to \Box B}$ 
\end{tabbing}
Note that axiom (IK4) is the same as (FS1) and (IK5) equals (FS3), and (FS2) and (FS4)
are analogous axioms for $\DIAMOND$ and $\BOX$. Note also that in \cite{simpson1994proof} it is argued that 
{\sf IK} is the true intuitionistic analogue of ``classical'' {\sf K}. 

\begin{proposition} \label{Prop:FSEquiv}
The following assertions hold in {\sf Int2GC}.
\begin{enumerate}[\rm (a)]
\item Axioms {\rm (FS1)} and {\rm (FS4)} are equivalent.
\item Axioms {\rm (FS2)} and {\rm (FS3)} are equivalent.
\end{enumerate}
\end{proposition}

\begin{proof}
We prove only assertion (a), because (b) can be proved analogously. Here
$\vdash A$ denotes that $A$ is provable in {\sf Int2GC}.

\smallskip
\noindent%
(FS1)$\Rightarrow$(FS4):  Let us set $X := A$, $Y := \Box \DIAMOND A$ and $Z := \Diamond \BOX B$ 
in the provable formula  $(X \to Y) \to ((Y \to Z) \to (X \to Z))$.
We get $\vdash (\Box \DIAMOND A \to \Diamond \BOX B) \to (A \to \Diamond \BOX B)$
by using also $\vdash A \to \Box \DIAMOND A$. This is equivalent 
to $\vdash A \wedge (\Box \DIAMOND A \to \Diamond \BOX B) \to \Diamond \BOX B$. 
Because $\vdash \Diamond \BOX B \to B$, this means  
$\vdash A \wedge (\Box \DIAMOND A \to \Diamond \BOX B)  \to B$ and
$\vdash (\Box \DIAMOND A \to \Diamond \BOX B)  \to(A \to B)$. 
If we set $A:= \DIAMOND A$ and $B:= \BOX B$ in (FS1), we obtain
$\vdash \Diamond(\DIAMOND A \to \BOX B) \to (\Box \DIAMOND A \to \Diamond \BOX B)$,
and so $\vdash \Diamond(\DIAMOND A \to \BOX B)  \to (A \to B)$. This implies
$\vdash (\DIAMOND A \to \BOX B)  \to \BOX(A \to B)$ by (GC\,${\Diamond}{\BOX}$).

\smallskip\noindent%
(FS4)$\Rightarrow$(FS1): 
We set $X := \DIAMOND \Box A$, $Y := A$ and $Z := B$ 
in $(X \to Y) \to ((Y \to Z) \to (X \to Z))$. This gives
$\vdash  (\DIAMOND \Box A \to A) \to ((A \to B) \to (\DIAMOND \Box A \to B))$,
and $\vdash (A \to B) \to (\DIAMOND \Box A \to \BOX \Diamond B)$,
since $\vdash \DIAMOND \Box A \to A$ and $\vdash B \to \BOX \Diamond B$.
By monotonicity, $\vdash \Diamond(A \to B) \to \Diamond(\DIAMOND \Box A \to \BOX \Diamond B)$. 
By setting $A := \Box A$ and  $B:= \Diamond B$ in (FS4), we have
$\vdash (\DIAMOND \Box A \to \BOX \Diamond B) \to \BOX (\Box A \to \Diamond B)$ and
$\vdash \Diamond(\DIAMOND \Box A \to \BOX \Diamond B) \to (\Box A \to \Diamond B)$ 
by (GC\,${\BOX}{\Diamond}$). Therefore, we obtain 
$\vdash \Diamond(A \to B) \to  (\Box A \to \Diamond B)$. 
\end{proof}

Logic {\sf Int2GC+FS} is defined as the extension of {\sf Int2GC} that satisfies (FS1) and (FS2). 
By Proposition~\ref{Prop:FSEquiv}, it is clear that we have several equivalent 
axiomatisations of {\sf Int2GC+FS} given in the next corollary.
\begin{corollary}
\[ \mbox{\sf Int2GC+FS} = {\sf Int2GC} \PLUS \{ {\rm (FS1) \ \text{or} \ (FS4)} \}
\PLUS \{ {\rm (FS2)  \ \text{or} \ (FS3)} \}. \]
\end{corollary}
\medskip

Logic {\sf Int2GC+FS} satisfies the counterparts of axioms 
(IK1)--(IK5) of {\sf IK}, so {\sf Int2GC+FS} can be regarded as a intuitionistic bi-modal logic, 
and the pairs of operators ($\Diamond$, $\Box$) and  ($\DIAMOND$, $\BOX$) can be regarded as intuitionistic modal 
connectives in the sense of Fischer Servi. Hence, {\sf Int2GC+FS} can be seen as an extension of
the fusion {\sf IK}\FUSE{\sf IK} of two copies of intuitionistic modal logic {\sf IK}, the first one with the modalities 
$({\Diamond},{\Box})$  and the second one with $({\DIAMOND},{\BOX})$. 

In ordered sets, there is another way of defining Galois connections presented in conditions \eqref{EQ:GC1} and \eqref{EQ:GC2}.
This gives us method (B) mentioned in Introduction. Let us considerer the fusion {\sf IK}\FUSE{\sf IK} of two copies of intuitionistic modal 
logic {\sf IK}, the first one with the operators $({\Diamond},{\Box})$ and the second one with $({\DIAMOND},{\BOX})$. We extend  
{\sf IK}\FUSE{\sf IK} by the so-called \emph{Brouwerian axioms}: 
\begin{tabbing}\TABS\label{Eq:Brower}%
(BR1) \ $A \to \BOX \Diamond A$ \>(BR2) \ $\Diamond \BOX A \to A$  \\
(BR3) \ $A \to \Box \DIAMOND A$ \>(BR4) \ $\DIAMOND \Box A \to A$  
\end{tabbing}
These axioms are also referred to as the \emph{converse axioms}, since these axioms are needed to ensure that the accessibility relations 
for the operators  $F,G$ and $P,H$ are each other’s converse in tense logics. We denote this logic by {\sf IK}\FUSE{\sf IK+BR}.

\begin{proposition} \label{Prop:equivalence}
{\sf Int2GC+FS} = {\sf IK}\FUSE{\sf IK+BR}.
\end{proposition}

\begin{proof} We have already noted that in {\sf Int2GC+FS} axioms (IK1)--(IK5) are provable for the operator pairs
$\Box$, $\Diamond$ and $\BOX$, $\DIAMOND$, and the operators $\Box$, $\Diamond$, $\BOX$, $\DIAMOND$ satisfy
the monotonicity rule. Additionally, Brouwerian axioms are provable in {\sf Int2GC+FS}.

On the other hand, $\vdash A \to \BOX B$ implies $\vdash \Diamond A \to \Diamond \BOX B$, which by (BR2) gives $\vdash \Diamond A \to B$. 
Similarly, $\vdash \Diamond A \to B$ implies $\vdash \BOX \Diamond A \to \BOX B$, and  by (BR1)  we get $\vdash A \to \BOX B$.
Thus, $(\Diamond, \BOX)$ is a Galois connection, and similarly we can show the same for the pair $(\DIAMOND, \Box)$. Fischer Servi
axioms (FS1)--(FS4) hold trivially in {\sf IK}\FUSE{\sf IK}.
\end{proof}

Proposition~\ref{Prop:equivalence} means that there exist two ways to extend intuitionistic logic with two Galois
connections such that these pairs are interlinked with Fischer Servi axioms. 

\begin{remark} \label{Rem:DefiningGalois}
The proof of Proposition~\ref{Prop:equivalence} reveals also that it is 
possible to endow a Galois connection in two ways to any logic $\mathcal{L}$ having modus ponens and 
satisfying the so-called \emph{law of syllogism} $(A \to B) \to ((B \to C) \to (A \to C))$.
The first way is to add operators $\Diamond$ and $\BOX$ to $\mathcal{L}$
and add rules (GC\,${\BOX}{\Diamond}$) and (GC\,${\Diamond}{\BOX}$). Or equivalently,
we may add Brouwerian axioms (BR1), (BR2) and rules of monotonicity (RM$\DIAMOND$) and (RM$\BOX$) for $\Diamond$ and $\BOX$.
Hence, the following are equivalent:
\begin{enumerate} [\rm (i)]
 \item $\mathcal{L} \PLUS \{ {\rm (GC}\,{\BOX}{\Diamond}), ({\rm GC}\,{\Diamond}{\BOX})  \}$
 \item $\mathcal{L} \PLUS \{ {\rm (BR1)}, {\rm (BR2)}, ({\rm RM} \Diamond), ({\rm RM} \BOX)   \}$
\end{enumerate} 
\end{remark}

Our next aim is to show that  {\sf Int2GC+FS} is equivalent to {\sf IK$_t$}. We need the following lemma.

\begin{lemma}\label{Lem:NewProperties}
The following formulas are {\sf Int2GC+FS}-provable:
\begin{enumerate}[\rm (a)]
 \item $\Box A \wedge \Diamond B \to \Diamond (A \wedge B)$ \quad and \quad $\BOX A \wedge \DIAMOND B \to \DIAMOND (A \wedge B)$;
 \item $\Box(A \to B) \to (\Diamond A \to \Diamond B)$  \quad and \quad  $\BOX (A \to B) \to (\DIAMOND A \to \DIAMOND B)$;
 \item $\Box \neg A \to \neg \Diamond A$ \quad and \quad  $\BOX \neg A \to \neg \DIAMOND A$.
\end{enumerate}
\end{lemma}

\begin{proof}
We only prove the first formula of each statement.

(a) Axiom (FS1) is equivalent to $\Diamond(A \to B) \wedge \Box A \to \Diamond B$. If we set $B := A \wedge B$ in this formula,  
we have that $\vdash (\Diamond(A \to A \wedge B) \wedge \Box A ) \to \Diamond (A \wedge B)$. 
Because $A \to A \wedge B$ is equivalent to $A \to B$, and $\vdash B \to (A \to B)$ and monotonicity of $\Diamond$ imply
$\vdash \Diamond B \to \Diamond(A \to B)$,  we obtain $\vdash \Diamond B \wedge \Box A \to \Diamond(A \wedge B)$.

(b) Because $\vdash \DIAMOND \Box (A \to B) \to (A \to B)$, we have $\vdash \DIAMOND \Box (A \to B) \wedge A \to B$
and $\vdash \Diamond (\DIAMOND \Box (A \to B) \wedge A)  \to \Diamond B$.
Let us set $A := \DIAMOND \Box(A \to B)$ and $B := A$ in (a). 
We obtain $\vdash \Box \DIAMOND \Box (A \to B) \wedge \Diamond A \to \Diamond (\DIAMOND \Box (A \to B) \wedge A)$. 
Thus, $\vdash \Box \DIAMOND \Box (A \to B) \wedge \Diamond A \to \Diamond B$.
Because $\vdash \Box (A \to B) \to \Box \DIAMOND \Box(A \to B)$, we have 
$\vdash \Box(A \to B) \wedge \Diamond A \to \Diamond B$. This is equivalent to
$\vdash \Box(A \to B) \to (\Diamond A \to \Diamond B )$.

(c) If we set $B := \bot$ in (b), we get $\vdash \Box(A \to \bot) \to (\Diamond A \to \Diamond \bot )$. 
Because $\Diamond \bot$ is equivalent to $\bot$, we obtain $\vdash \Box \neg A \to \neg \Diamond A$.
\end{proof}

Next, we show that {\sf IK$_t$} and {\sf Int2GC+FS} are syntactically equivalent. Logic 
{\sf IK$_t$} is obtained by extending the language of intuitionistic propositional logic with the usual temporal expressions 
$FA$ ($A$ is true at some future time),
$PA$ ($A$ was true at some past time), 
$GA$ ($A$ will be true at all future times), and
$HA$ ($A$ has always been true in the past).
The following  Hilbert-style axiomatisation of {\sf IK$_t$} is given by Ewald in \cite{Ewald86}*{p.~171}:
\begin{tabbing}
\TABS
\ (1) \ All axioms of intuitionistic logic \\
\ (2) \ $G(A \to B) \to (GA \to GB)$ 				\>\ (2$'$) \ $H(A \to B) \to (HA \to HB)$ \\
\ (3) \ $G(A \wedge B) \leftrightarrow GA \wedge GB$ 		\>\ (3$'$) \ $H(A \wedge B) \leftrightarrow HA \wedge HB$ \\  
\ (4) \ $F(A \vee B) \leftrightarrow FA \vee FB$ 		\>\ (4$'$) \ $P(A \vee B) \leftrightarrow PA \vee PB$ \\
\ (5) \ $G(A \to B) \to (FA \to FB)$ 				\>\ (5$'$) \ $H(A \to B) \to (PA \to PB)$ \\ 
\ (6) \ $GA \wedge FB \to F(A \wedge B)$			\>\ (6$'$) \ $HA \wedge PB \to P(A \wedge B)$ \\
\ (7) \ $G \neg A  \to \neg FA$					\>\ (7$'$) \ $H \neg A \to \neg PA$ \\
\ (8) \ $FHA \to A$						\>\ (8$'$) \ $PGA \to A$ \\
\ (9) \ $A \to HFA$ 						\>\ (9$'$) \ $A \to GPA$ \\
$\!$(10) \ $(FA \to GB) \to G(A \to B)$				\>$\!$(10$'$) \ $(PA \to HB) \to H(A \to B)$ \\
$\!$(11) \ $F(A \to B) \to (GA \to FB)$				\>$\!$(11$'$) \ $P(A \to B) \to (HA \to PB)$ \\
\end{tabbing}
The rules of inference are modus ponens (MP), and 
\begin{tabbing}
\TABS
(RH) \ $\displaystyle \frac{A}{H A}$
\>$\!$(RG) \ $\displaystyle \frac{A}{G A}$
\end{tabbing}\medskip

Our next theorem shows that if we identify $\Diamond$, $\Box$, $\DIAMOND$, $\BOX$ with
$F$, $G$, $P$, $H$, respectively, then {\sf Int2GC+FS} and {\sf IK$_t$} are  syntactically equivalent.

\begin{theorem}\label{Thm:Equivalence} $\text{\sf IK}_t = \text{\sf Int2GC+FS}$.
\end{theorem}

\begin{proof} 
First we will show that the {\sf IK$_t$}-axioms are provable in {\sf Int2GC+FS}, and all rules
of {\sf IK$_t$} are admissible in {\sf Int2GC+FS}. As mentioned in Section~\ref{Sec:Axiomatisations},
axioms (2), (2$'$), (3), (3$'$) (4), (4$'$), (8), (8$'$), (9), (9$'$) are provable even in {\sf Int2GC}.
Additionally, rules (MP), (RH), and (RG) are admissible in {\sf Int2GC}.
Axioms (10), (10$'$), (11), (11$'$) are Fischer Servi axioms (FS3), (FS4), (FS1), (FS2), so
they are provable in {\sf Int2GC+FS}. The provability of (5), (5$'$), (6), (6$'$), (7), and (7$'$)                                                      
is shown in Lemma~\ref{Lem:NewProperties}.

Because axioms (10), (10$'$), (11), (11$'$) are the Fischer Servi axioms, for the other direction it
is enough to show the admissibility of rules (GC\,${\BOX}{\Diamond}$), (GC\,${\Diamond}{\BOX}$), (GC\,${\Box}{\DIAMOND}$), 
(GC\,${\DIAMOND}{\Box}$) in {\sf IK$_t$}. First, we show the admissibility of the rules of monotonicity, 
that is, if $A \to B$ is provable, then $H A \to H B$, $PA \to P B$, $G A \to G B$, and $FA \to F B$ are provable.

Here $\vdash A$ denotes that the formula $A$ is provable in {\sf IK$_t$}.
Assume  $\vdash A \to B$. By (RG), $\vdash G(A \to B)$.
Now $\vdash G A \to G B$ follows by (2), and from $\vdash G(A \to B)$,
we obtain also $\vdash FA \to FB$ by (5). Similarly, $\vdash A \to B$ implies 
$\vdash HA \to H B$ and $\vdash P A \to P B$ by applying (RH), (2$'$), and (5$'$).

Next we prove the admissibility of (GC\,${\BOX}{\Diamond}$). Assume that  $\vdash A \to H B$. 
Then,  $F A \to F H B$ by the monotonicity of $F$. Because $\vdash F H B \to B$ by (8),
we obtain $\vdash F A \to B$.  Similarly, by (8$'$) and the monotonicity of $P$, $A \to GB$ implies 
$PA \to B$, that is,  (GC\,${\Box}{\DIAMOND}$) is admissible in {\sf IK$_t$}.
The monotonicity of $H$ and axiom (9) yield that $FA \to B$ implies $A \to HB$,
and monotonicity of $G$ and (9$'$) give that $PA \to B$ implies $A \to BG$. Thus, rules
(GC\,${\Diamond}{\BOX}$) and (GC\,${\DIAMOND}{\Box}$) are admissible. 
\end{proof}

By combining Proposition~\ref{Prop:equivalence} and Theorem~\ref{Thm:Equivalence}, we get the following corollary.
Note that {\sf IK$_t$} = {\sf IK}\FUSE{\sf IK+BR} is proved already by Davoren \cite{Davoren}.

\begin{corollary} \label{Prop:Int equivalence}
{\sf IK$_t$} = {\sf Int2GC+FS} = {\sf IK}\FUSE{\sf IK+BR}.
\end{corollary}

In \cite{JaKoKo08}, the logic extending classical logic {\sf Cl} with a Galois connection $(\Diamond, \BOX)$
was introduced and it is proved  that if we add another two operators $\DIAMOND$ and $\Box$
that are connected to the Galois connection $(\Diamond, \BOX)$ by the De~Morgan axioms:
\begin{tabbing} \TABS
(DM1) \ $\Box A \leftrightarrow \neg \Diamond \neg A$ \> (DM2) $\DIAMOND A \leftrightarrow \neg \BOX \neg A$,
\end{tabbing}
then also the pair $(\DIAMOND,\Box)$ is a Galois connection, that is, rules (GC\,${\Box}{\DIAMOND}$) 
and (GC\,${\DIAMOND}{\Box}$) are admissible. Hence, in classical case,  one Galois connection is defined by the other (obtained ``for free''), 
which is not the case in intuitionistic logic; see \cite{DzJaKo10}.
It is proved in \cite{JaKoKo08} that this logic is syntactically equivalent to classical tense logic {\sf K$_t$}, when $\Diamond$, $\Box$, $\DIAMOND$, $\BOX$ are identified
with the tense operators $F$, $G$, $P$, $H$, respectively. Note that algebras corresponding to {\sf K$_t$} are considered in 
\cite{vonKarger95,Venema2007}, for example, and these are generally called  \emph{tense algebras}.

As stated in \cite[p.~54]{simpson1994proof}, it is routine to derive $\Diamond A \leftrightarrow \neg \Box \neg A$ in {\sf IK},  together with the Law of the Excluded Middle. Since {\sf Int2GC+FS} is an extension of the fusion {\sf IK}\FUSE{\sf IK} of two copies of intuitionistic modal logic {\sf IK},
then it is clear that classical logic with two Galois connection pairs $(\Diamond, \BOX)$ and $(\DIAMOND,\Box)$, which
are interlinked with (FS1) and (FS2), denoted here {\sf Cl2GC+FS},  satisfies (DM1) and (DM2). On the other hand, in {\sf K$_t$}, the pairs $(F,H)$ and $(P,G)$ 
form Galois connections, and axioms (FS1) and (FS2) are provable. Therefore, we can write:
\begin{center}
  {\sf K$_t$} = {\sf Cl2GC+FS}
\end{center}
Observe that {\sf K$_t$} can be defined as the fusion {\sf K}\FUSE{\sf K} extended with Brouwerian
axioms (BR1)--(BR4), denoted by {\sf K}\FUSE{\sf K+BR}. In summary, we have: 

\begin{corollary} \label{Prop:Cl equivalence}
{\sf K$_t$} = {\sf Cl2GC+FS} = {\sf K}\FUSE{\sf K+BR}.
\end{corollary}

In conclusion, if we add to intuitionistic logic {\sf Int} 
two Galois connections $(\Diamond, \BOX)$ and $(\DIAMOND, \Box)$ that are connected using
Fischer Servi axioms (FS1) and (FS2), then we get the intuitionistic tense logic
{\sf IK$_t$}. Analogously, if two Galois connections combined with axioms 
(FS1) and (FS2) are added to classical logic, we obtain the classical tense logic {\sf K$_t$}. Here we discuss how
for each intermediate logic {\sf L}, we can define the corresponding \emph{{\sf L}-tense logic} {\sf LK$_t$}.

An \emph{intermediate logic} is a propositional logic extending intuitionistic logic.
Classical logic {\sf Cl} is the strongest intermediate logic and it is obtained from
{\sf Int} by extending the axioms of {\sf Int} by the ``Law of the excluded middle'' $A \vee \neg A$,
or equivalently, by the ``Double negation elimination'' $\neg\neg A \to A$ or by 
``Peirce's law'' $((A \to B) \to A) \to A$. 
There exists a continuum of different intermediate logics. For example, the G\"{o}del--Dummett
logic {\sf G} is obtained from {\sf Int} by adding the axiom $(A \to B) \vee  (B \to A)$. 
For more examples of intermediate logics and their semantics; see \cite{ChaZak97,GhiMig99}.

We  denote by {\sf L} any intermediate logic, that is, ${\sf Int} \subseteq {\sf L} \subseteq {\sf Cl}$.
We can write that for any intermediate logic {\sf L},
\[ \mbox{\sf IK}_t \subseteq \mbox{\sf L2GC+FS} \subseteq \mbox{\sf K}_t .\]
Because Ewald's {\sf IK$_t$} is commonly accepted as the intuitionistic analogue of the classical
tense logic {\sf K$_t$},  taking into account the equivalences $\text{\sf IK}_t = \text{\sf Int2GC+FS}$ and 
$\mbox{\sf K}_t = \mbox{\sf Cl2GC+FS}$, the logic {\sf L2GC+FS} can be regarded as the  
{\sf L}-tense logic {\sf LK$_t$} for any intermediate logic {\sf L}.  Then, as one of the main results of this work,
we have a general uniform method (A) of building {\sf L}-tense logic for any intermediate logic {\sf L} by 
setting $\mbox{\sf LK$_t$} = \mbox{\sf L2GC+FS}$, that is, {\sf L} added with two Galois connection pairs
combined using Fischer Servi axioms. Moreover, since the equivalences  $\text{\sf K}_t = \text{\sf Cl2GC+FS}$ 
and $\text{\sf IK}_t = \text{\sf Int2GC+FS}$ were shown syntactically, {\sf L}-tense logics exist 
independent of whether they are Kripke complete or canonical, or not. This can be presented as the following theorem.

\begin{theorem}
For any intermediate logic {\sf L}, there is an {\sf L}-tense logic {\sf LK$_t$}, which is {\sf L} endowed with two independent
Galois connections connected by Fischer Servi axioms {\rm (FS1)} and {\rm (FS2)}.
\end{theorem}
 
\section{Heyting algebras with Galois connections} \label{Sec:Algebras}

E.~Or{\l}owska and I.~Rewitzky \cite{OrlRew07} defined a \emph{Heyting algebra with modal operators} as
a Heyting algebra $(H,\vee,\wedge,\to,0,1)$ equipped with unary operators $\Diamond$ and $\Box$ satisfying
for all $x,y \in H$:
\begin{equation} \label{EQ:Additive}
\Diamond x  \vee \Diamond y = \Diamond(x \vee y)
\text{ \quad and \quad } 
\Box x \wedge \Box y = \Box (x \wedge y).
\end{equation}
These algebras are called \emph{HM-algebras}, for short. In addition, they defined \emph{HK-algebras} as
HM-algebras satisfying
\begin{equation} \label{Eq:Normal}
\Diamond 0 = 0
\text{ \quad and \quad }
\Box 1 = 1.
\end{equation}

We introduced in \cite{DzJaKo10} \emph{HGC-algebras} as Heyting algebras provided with an order-preserving 
Galois connection $(\Diamond, \BOX)$. Equationally HGC-algebras can be defined as algebras
$(H,\vee,\wedge,\to,0,1,\Diamond, \BOX)$ such that $(H,\vee,\wedge,\to,0,1)$ satisfies the
identities for Heyting algebras (which can be found in e.g.\@ \cites{BaDw74, BuSa81}), the
operators $\Diamond$ and $\BOX$ satisfy \eqref{EQ:Additive}, and for all $x \in H$,
\begin{equation} \label{Eq:Connectivity}
x \leq \BOX \Diamond x 
\text{ \quad and \quad }
\Diamond \BOX x \leq x.
\end{equation}
By definition, HGC-algebras are HM-algebras, but HGC-algebras are also HK-algebras, because
$0 \leq \BOX 0$ implies $\Diamond 0 \leq 0$ and $\Diamond 1 \leq 1$ gives $1 \leq \BOX 1$.
Thus, $\Diamond$ and $\BOX$ satisfy \eqref{Eq:Normal}.

In \cite{DzJaKo10}, we proved that {\sf IntGC} is algebraizable in terms of HGC-algebras.
More precisely, any valuation $v$ assigning to propositional variables elements of an
HGC-algebra can be extended to all formulas inductively by the following way:
\begin{align*}
 v(A\wedge B) &= v(A) \wedge v(B)  & v(A \vee B) &= v(A) \vee v(B) \\
 v(A \to B) &= v(A) \to v(B)       & v(\neg A) &= \neg v(A) \\
 v(\Diamond A) &= \Diamond v(A)    & v(\BOX A) &= \BOX v(A).
\end{align*}
Then, a formula $A$ is provable in {\sf IntGC} if and only if $v(A) = 1$ for all
valuations $v$ on any HGC-algebra.

\medskip

We define \emph{H2GC-algebras} as structures $(H,\vee,\wedge,\to,0,1, \Diamond, \Box, \DIAMOND, \BOX)$
such that $(H,\vee,\wedge,\to,0,1,\Diamond, \BOX)$ and $(H,\vee,\wedge,\to,0,1,\DIAMOND, \Box)$  
are HGC-algebras. Similarly as in case of {\sf IntGC}, we can show, by applying Lindenbaum--Tarski
algebras, that {\sf Int2GC} is complete with respect to {H2GC}-algebras, that is,
a formula $A \in \Phi$ is provable in {\sf Int2GC} if and only if $v(A) = 1$  for all
valuations $v$ on any H2GC-algebra.

Or{\l}owska and Rewitzky \cite{OrlRew07} studied also an extension of HK-algebras, called
\emph{HK1-algebras}, that are algebraic counterparts of the logic {\sf IK}. They extended HK-algebras
by the following two conditions that correspond to Fischer Servi axioms (FS1) and (FS2):
\begin{equation}
 \Diamond( x \to\ y) \leq \Box x \to \Diamond y \text{ \quad and \quad } 
 \Diamond x \to \Box y \leq \Box(x \to y) 
\end{equation}
Let us denote for any H2GC-algebra  $(H,\vee,\wedge,\to,0,1, \Diamond, \Box,\DIAMOND, \BOX)$ 
the corresponding conditions by (FS1)--(FS4), that is,
\begin{tabbing}\TABS
({FS}1) \ $\Diamond(x \to y) \leq \Box x \to \Diamond y$  \>({FS}2) \ $\Diamond x \to \Box y  \leq \Box(x \to y)$  \\
({FS}3) \ $\DIAMOND(x \to y) \leq \BOX x \to \DIAMOND y$  \>({FS}4) \ $\DIAMOND x \to \BOX y  \leq \BOX(x \to y)$ 
\end{tabbing}
Note that we used  (FS1)--(FS4) to denote also the corresponding Fischer Servi axioms in logic. This
should not cause any confusion, because the context shows whether we are dealing with logic
axioms or lattice-theoretical conditions.
In addition, we denote by (D1) and  (D2) the conditions corresponding the first condition
of \eqref{Eq:Dunn}, that is,
\begin{tabbing}\TABS
(D1) \  $\Diamond x \wedge \Box y \leq \Diamond( x \wedge y)$
\> (D2) \  $\DIAMOND x \wedge \BOX y \leq \DIAMOND( x \wedge y)$
\end{tabbing}

\begin{proposition} \label{Prop:DunnConditions}
Let   $(H,\vee,\wedge,\to,0,1, \Diamond, \Box,\DIAMOND, \BOX)$ be an H2GC-algebra.
\begin{enumerate}[\rm (a)]
\item Conditions {\rm (FS1)}, {\rm (D1)}, and {\rm (FS4)} are equivalent.
\item Conditions {\rm (FS2)}, {\rm (D2)}, and {\rm (FS3)}  are equivalent.
\end{enumerate}
\end{proposition}

\begin{proof} We prove only assertion (a), because (b) can be proved in a similar way.
Assume that (D1) holds, and set $x := a \to b$ and $y := a$ in it. We obtain 
$\Diamond (a \to b) \wedge  \Box a  \leq \Diamond (a \wedge (a \to b)) \leq \Diamond b$,
because $a \wedge (a \to b) \leq b$. This gives directly 
$\Diamond (a \to b) \leq \Box a  \to \Diamond b$, and thus
(D1) implies (FS1). Conversely, if we set $x:=b$ and $y := a \wedge b$ in (FS1), we have
$\Diamond a \leq \Diamond (b \to a) = \Diamond (b \to a \wedge b) \leq \Box b \to \Diamond (a \wedge b)$,
because $b \to a \wedge b = b \to a$ and $a \leq b \to a$.
This is equivalent to $\Diamond a \wedge \Box b \leq \Diamond (a \wedge b)$.
Hence, also (FS1) implies (D1).
That (FS1) and (FS4) are equivalent can be shown as in Proposition~\ref{Prop:FSEquiv}.
\end{proof}

Proposition~\ref{Prop:DunnConditions} together with the completeness of {\sf Int2GC} with respect to H2GC-algebras implies that 
\[ \mbox{\sf IK$_t$} = \mbox{\sf Int2GC+FS} = {\sf Int2GC} \PLUS \{ {\rm (D1)} , {\rm (D2)} \}, \]
where (D1) and (D2) denote the axioms:
\begin{tabbing}\TABS
(D1) \ $\Diamond A \wedge \Box B \to \Diamond(A \wedge B)$ 
\>(D2) $\DIAMOND A \wedge \BOX B \to \DIAMOND(A \wedge B)$.
\end{tabbing}

Let us define \emph{H2GC+FS-algebras} as H2GC-algebras satisfying (FS1) and (FS2). Proposition~\ref{Prop:DunnConditions}
has the following corollary.

\begin{corollary}\label{Cor:CharacteringAlgebra} 
Let\/ $\mathbb{H} = (H,\vee,\wedge,\to,0,1, \Diamond, \Box,\DIAMOND, \BOX)$ be an H2GC-algebra.

\begin{enumerate}[\rm (a)]
 \item $\mathbb{H}$ is an H2GC+FS-algebra if and only if  $(H,\vee,\wedge,\to,0,1, \Diamond, \Box)$ and
 $(H,\vee,\wedge,\to,0,1,\DIAMOND, \BOX)$ are HK1-algebras. 
 \item $\mathbb{H}$ is an H2GC+FS-algebra if and only if it satisfies {\rm (D1)} and {\rm (D2)}.
\end{enumerate}
\end{corollary}

For any H2GC+FS-algebra  $(H,\vee,\wedge,\to,0,1, \Diamond, \Box,\DIAMOND, \BOX)$, a
valuation $v$ is a function $v \colon \mathit{Var} \to H$, which is inductively extended
to all formulas in $\Phi$ as is done above in the case of HGC-algebras. A formula
$A \in \Phi$ is \emph{H2GC+FS-valid} if $v(A) = 1$ for every valuation $v$ on
any H2GC+FS-algebra.

We have shown in \cite{DzJaKo10} that rules (GC\,${\BOX}{\Diamond}$) and (GC\,${\Diamond}{\BOX}$) 
preserve validity, and obviously the same holds for  (GC\,${\Box}{\DIAMOND}$) and (GC\,${\DIAMOND}{\Box}$).
In addition, axioms (FS1) and (FS2) are also valid, because H2GC+FS-algebras are defined by
using analogous conditions. Thus, {\sf Int2GC+FS}-provable formulas are H2GC+FS-valid.

To obtain algebraic completeness, we apply Lindenbaum--Tarski algebras. 
We denote by $\mathcal{F}(\Phi)$ the \emph{algebra of $\Phi$-formulas}, that is, 
the abstract algebra 
\[ \mathcal{F}(\Phi) = (\Phi,\vee,\wedge,\to,\bot,{\Diamond},{\Box},{\DIAMOND}, {\BOX}).\]
We define an equivalence $\equiv$ on $\Phi$ by
\[
A \equiv B  \iff A \leftrightarrow B \text{ is {\sf Int2GC+FS}-provable}.
\]
It is easy to observe that $\equiv$ is a congruences on $\mathcal{F}(\Phi)$. Let  $[A]$ denote the $\equiv$-class of $A$.
We define the \emph{quotient algebra} $\mathcal{F}(\Phi)/{\equiv}$ by introducing the operations:
\begin{align*}
&{[A]} \vee [B] =  [A \vee B], \quad {[A]} \wedge [B]  =  [A \wedge B], \quad {[A]} \to [B]  =  [A \to B],    \\
&{\Diamond [A]} =  [\Diamond A], \quad \Box {[A]} = [\Box A], \quad {\DIAMOND [A]} = [\DIAMOND A], \quad {\BOX [A]}  = [\BOX A] 
\end{align*}

Because H2GC+FS-algebras form an equational class, $\mathcal{F}(\Phi)/{\equiv}$ forms an H2GC+FS-algebra.
Note that $[\bot]$ and $[\top]$ are the zero and the unit in this algebra. We define a valuation
$v \colon \mathit{Var} \to \Phi / {\equiv}$ by $v(p) = [p]$. By straightforward formula induction, we see that 
$v(A) = [A]$ for all formulas $A \in \Phi$. If now $A \in \Phi$ is H2GC+FS-valid, then $v(A) = [\top]$
in  $\mathcal{F}(\Phi)/{\equiv}$. This means $A \leftrightarrow \top$ and thus $A$ is
{\sf Int2GC+FS}-provable. Therefore, we can write the following completeness theorem.

\begin{theorem} \label{Thm:CompletenessI}
A formula $A \in \Phi$ is {\sf Int2GC}-provable if and only if $A$ is H2GC+FS-valid.
\end{theorem}

If we change the underlying logic from intuitionistic to classical, we have
that {\sf Cl2GC+FS} is complete with respect to tense algebras -- this is due to the
standard algebraic completeness theorem of temporal logic {\sf K$_t$} with respect 
to tense algebras.

Results of this section can be equally applied to intermediate logics. It is well known that intuitionistic logic and all 
intermediate logics are algebraizable; see, for example, \cite{BlokPig89}. For instance, the specific axiom $(A \to B) \vee  (B \to A)$
of G{\"o}del--Dummett logic {\sf G} translates into in the identity $(x \to y) \vee  (y \to x) = 1$ extending Heyting algebras. 
For every intermediate logic {\sf L}, there exists a corresponding equational class of {\sf L}-\emph{algebras}.
For each {\sf L}-algebra $(H_{\sf L},\vee,\wedge,\to,0,1)$, we define the corresponding \emph{H$_{\sf L}$2GC+FS-algebra} as an algebra
$(H_{\sf L},\vee,\wedge,\to,0,1,\Diamond,\Box,\DIAMOND,\BOX)$ by using the same identities as in the case of defining H2GC+FS-algebras
from Heyting ones. Clearly, the class of H$_{\sf L}$2GC-algebras is equational.
Since the method of Lindenbaum--Tarski algebras is applicable to any {\sf L2GC+FS}-logic in a straightforward way, we get the algebraic completeness. 
 
\begin{corollary} \label{:CompletenessI} For every intermediate logic {\sf L}, 
a formula $A \in \Phi$ is {\sf L2GC+FS}-provable if and only if $A$ is valid in every H$_{\sf L}$2GC--algebra.
\end{corollary}

Very often completeness for an intermediate logic {\sf L} is stated for a narrower class than the class of all {\sf L}-algebras.
For instance,  G\"{o}del--Dummett logic {\sf G} is complete with respect to the class of finite chains, and in \cite{DzJaKo12},
we showed the finite model property of {\sf IntGC}. However, here we will not consider algebraic completeness of {\sf L2GC+FS}-logics 
with respect to these kinds of narrower classes. 

Let $\Phi_0$ denote the set of propositional formulas of intuitionistic logic only (thus not containing
$\Diamond, \Box,\DIAMOND, \BOX$). In  \cite[Prop.~4.6]{DzJaKo10}, we proved that {\sf IntGC} is conservative over
{\sf Int}, and analogously we can prove the following theorem.

\begin{theorem} \label{Thm:Conservativeness} 
For every intermediate logic {\sf L},
a formula $A \in \Phi_0$ is {\sf L2GC+FS}-provable if and only if $A$ is provable in intermediate propositional logic {\sf L}.
\end{theorem}

\begin{example} \label{Exa:Motivation}
(a) As a motivating example for H2GC+FS-algebras, we consider fuzzy modal operators on complete 
Heyting algebras. These are also closely connected to fuzzy Galois connections (see e.g.\@ \cites{Belo99,GeoPop04}). 

A \emph{complete Heyting algebra} is a Heyting algebra $(H,\vee,\wedge,\to,0,1)$
such that its underlying lattice $(H,\leq)$ is complete.  It is well known \cites{Grat98,RasSik68}
that a complete Heyting algebra $H$ satisfies the \emph{join-infinite distributive law}:  
for any $S \subseteq H$ and $x \in H$,
$x \wedge \left ( \bigvee S \right ) = \bigvee \{ x \wedge y \mid y \in S \}$.

Fuzzy sets on complete Heyting algebras generalise fuzzy sets on the  unit interval $[0,1]$.
Let $U$ be some universe of discourse. Each map $\varphi \colon U \to H$ is called
a \emph{fuzzy set} on $U$. For any object $x$, $\varphi(x)$ is the
\emph{grade of membership}. We denote by $H^U$ the set of all fuzzy sets
on $U$. Also $H^U$ forms a complete Heyting algebra in which the operations
are defined pointwise.

Let $R$ be a fuzzy relation on $U$, that is, $R$ is a mapping from $U \times U$ to $H$.
For a fuzzy set $\varphi \in H^U$, we may define the fuzzy sets
$\Diamond \varphi$, $\Box \varphi$, $\DIAMOND \varphi$, $\BOX \varphi$ by
setting for all $x \in U$:
\begin{align*}
\Diamond \varphi (x) & =  \bigvee_{y \in U} \{ R(x,y) \wedge \varphi(y) \} 
& \Box \varphi(x) & =  \bigwedge_{y \in U} \{  R(x,y) \to \varphi(y) \} \\
\DIAMOND \varphi (x) & =  \bigvee_{y \in U} \{ R(y,x) \wedge \varphi(y) \} 
& \BOX \varphi (x) & =  \bigwedge_{y \in U} \{  R(y,x) \to \varphi(y) \}
\end{align*}

We show first that $({\Diamond},{\BOX})$ and $({\DIAMOND},{\Box})$
are Galois connections on $H^U$. Indeed, suppose $\varphi$ and $\psi$ are fuzzy sets such that
$\varphi \leq \psi$. Then, for all $y \in U$, $R(x,y) \wedge \varphi(y) \ \leq  \ R(x,y) \wedge \psi(y)$
and this implies
\[ \Diamond \varphi (x) = \bigvee_{y \in U}  \{ R(x,y) \wedge \varphi(y) \} \leq  
\bigvee_{y \in U}  \{ R(x,y) \wedge \psi(y) \} = \Diamond \psi (x). \]
Similarly, $ R(y,x) \to \varphi(y) \ \leq \ R(y,x) \to \psi(y) $ for all 
$y \in U$. Thus,
\[ \BOX \varphi (x) =  \bigwedge_{y \in U} \{  R(y,x) \to \varphi(y) \} \leq
\bigwedge_{y \in U} \{  R(y,x) \to \psi(y) \} =  \BOX \psi(x). \]
So, $\Diamond$ and $\BOX$ are order-preserving. By definition, for all $x \in U$, 
\begin{eqnarray*}
\Diamond \BOX \varphi(x) & = & \bigvee_{y \in U}  \{ R(x,y) \wedge \BOX \varphi (y)  \}  
=  \bigvee_{y \in U} \Big \{ R(x,y) \wedge  \bigwedge_{z \in U} \{  R(z,y) \to \varphi(z) \} \, \Big \} \\
& \leq & \bigvee_{y \in U} \{ R(x,y) \wedge  ( R(x,y) \to \varphi(x) \, ) \, \} 
 \leq  \bigvee_{y \in U} \{  \varphi(x) \} =  \varphi(x).
\end{eqnarray*}
This means that $\Diamond \BOX \varphi \leq \varphi$. 
Analogously, for any $x \in U$, 
\begin{eqnarray*}
\BOX \DIAMOND \varphi(x) & = & \bigwedge_{y \in U} \{  R(y,x) \to \Diamond \varphi (y) \} 
=  \bigwedge_{y \in U} \Big \{  R(y,x) \to \bigvee_{z \in U} \{ R(y,z) \wedge \varphi(z) \} \, \Big \} \\
& \geq & \bigwedge_{y \in U} \{  R(y,x) \to ( \, R(y,x) \wedge \varphi(x) \, ) \} 
\geq  \bigwedge_{y \in U} \{  \varphi(x) \} =  \varphi(x). 
\end{eqnarray*}
Thus, also $\varphi \leq \BOX \Diamond \varphi$. We have that  $({\Diamond},{\BOX})$
is a Galois connection. Similarly, we can show that  $({\DIAMOND},{\Box})$ is a Galois connection.

Next we show that (D1) holds. For all $x,y \in U$, we have
\begin{align*}
 R(x,y) \wedge \varphi(y)  \wedge \Box \psi(x) &=
 R(x,y) \wedge \varphi(y)  \wedge \bigwedge_{z \in U} \{  R(x,z) \to \psi (z) \}\\
 & \leq R(x,y) \wedge \varphi(y) \wedge (R(x,y) \to \psi(y)) \\
 & = (R(x,y) \wedge (R(x,y) \to \psi(y)))  \wedge \varphi(y) \\
 & = R(x,y) \wedge \psi(y)  \wedge \varphi(y) \\
 & = R(x,y) \wedge (\varphi \wedge \psi)(y)\\
 & \leq \bigvee_{z \in U} \{ R(x,z) \wedge (\varphi \wedge \psi)(z) \}\\
 & = \Diamond (\varphi \wedge \psi)(x).
\end{align*}
Hence, for all $x,y \in U$, $R(x,y) \wedge \varphi(y) \wedge \Box \psi(x) \leq  \Diamond (\varphi \wedge \psi)(x)$.
Because complete Heyting algebras satisfy the join-infinite distributive law, we have 
\begin{align*}
(\Diamond \varphi \wedge \Box \psi)(x) & = \Diamond \varphi(x) \wedge \Box \psi(x)  
= \bigvee_{y \in U} \{ R(x,y) \wedge \varphi(y)\} \wedge \Box \psi(x) \\
& = \bigvee_{y \in U} \{ R(x,y) \wedge \varphi(y) \wedge \Box \psi(x) \} 
\leq \Diamond (\varphi \wedge \psi)(x). 
\end{align*}
Thus, $\Diamond \varphi \wedge \Box \psi \leq  \Diamond (\varphi \wedge \psi)$.
Assertion (D2) can be proved similarly.

\medskip

(b) The instances
\begin{equation}
 \Box (a \vee b) \leq \Box a \vee \Diamond b \quad \text{ and } \quad \BOX(a \vee b) \leq \BOX a \vee \DIAMOND b
\end{equation}
of  Dunn's second axiom of \eqref{Eq:Dunn} are false in some H2GC+FS-algebras of fuzzy modalities.

Namely, let $U = \{x,y\}$ and consider the finite (and hence complete) Heyting algebra $\mathbf{2}^2 \oplus 1$, 
that is, $\mathbb{H} = \{0, a, b, c, 1\}$ is the Heyting algebra with the order 
$0 < a, b < c < 1$, where $a$ and $b$ are incomparable. Note that $\neg a = b$ and $\neg b = a$.

We define two fuzzy sets $\varphi, \psi$ on $U$ by setting $\varphi(u) = 0 $ and $ \psi(u) = 1$ for all $u \in U$. 
A fuzzy relation $R \colon U \times U \to H$ is defined by  $R(x, x) = R(y,y) = a$ and $R(x, y) = R(y,x) = b$. Then,
\[
\Box (\varphi \vee \psi) (x) = \bigwedge_{u \in U}  (R(x,u) \to (\varphi \vee \psi) (u)) = 1,
\]
but 
\begin{align*}
\Box \varphi (x) \vee  \Diamond \psi (x) &= \bigwedge_{u \in U}  (R(x,u) \to  \varphi (u)) \vee \bigvee_{u \in U} (R(x,u) \wedge \psi (u)) \\
&= (\neg a \wedge \neg b) \vee (a \vee b) = c.
\end{align*}
Hence, $\Box (\varphi \vee \psi) \leq \Box \varphi \vee \Diamond \psi$ is not satisfied. 
Similarly, $\BOX (\varphi \vee \psi) (y) = 1$ \ and \ $\BOX \varphi (y) \vee  \DIAMOND \psi (y) = c$,
that is,  $\BOX (\varphi \vee \psi) \leq \BOX \varphi \vee \DIAMOND \psi$ is not satisfied.

\medskip
(c) In Lemma~\ref{Lem:NewProperties}(c), we showed the provability of $\Box \neg A \to \neg \Diamond A$ and $\BOX \neg A \to \neg \DIAMOND A$.
This implies that also $\Diamond A \to \neg \Box \neg A$ and  $\DIAMOND A \to \neg \BOX \neg A$ are provable. Here we show that
De~Morgan axioms $\Diamond A \leftrightarrow \neg \Box \neg A$ and  $\Box A \leftrightarrow \neg \Diamond \neg A$ discussed in
Introduction are not provable in {\sf Int2GC+FS}.

Let us consider a linear Heyting algebra (that is, a G{\"o}del algebra)
\[ \mathbb{H} = \left \{ \frac{1}{n+1} \, \middle| \, n \in \mathbb{N} \right \} \cup \{ 0,1 \}, \]
where $\mathbb{N} = \{ 1,2,3,\ldots\}$.
Let us set $U = \mathbb{N}$ and define a fuzzy set $\varphi \colon U \to H$ by setting
$\varphi(n) = \frac{1}{n+1}$ for all $n \in \mathbb{N}$. We also define a fuzzy relation $R$ on $U$ simply by setting $R(m,n) = 1$ 
for all $m,n \in \mathbb{N}$.

For all $n \in \mathbb{N}$,
\[ \Diamond \varphi(n) = \bigvee_{m \in \mathbb{N}} \{ R(n,m) \wedge \varphi(m) \}
= \bigvee_{m \in \mathbb{N}} \Big \{ \frac{1}{m+1}  \Big \} = \frac{1}{2}, \] 
but 
\[ \Box \neg \varphi(n) =  \bigwedge_{m \in \mathbb{N}} \{ R(n,m) \to \neg \varphi(m) \} = \bigwedge_{m \in \mathbb{N}}  \{ 1 \to 0 \} = 0,\]
which means $\neg \Box \neg \varphi(n) = 1$ for all $n \in \mathbb{N}$. Thus, $\Diamond \varphi(n) \neq \neg \Box \neg \varphi(n)$ for all $n \in \mathbb{N}$. 
Similarly, we can show $\Box \varphi \neq \neg \Diamond \neg \varphi$. Indeed, for $n \in  \mathbb{N}$,
\[
 \Box \varphi(n) = \bigwedge_{m \in \mathbb{N}} \Big \{ R(n,m) \to  \frac{1}{m+1} \Big \} \\
		   = \bigwedge_{m \in \mathbb{N}} \Big \{ \frac{1}{m+1} \Big \} = 0.
\]
On the other hand,
\[
 \Diamond \neg \varphi(n) = \bigvee_{m \in \mathbb{N}} \{ R(n,m) \wedge \neg \varphi(m) \} 
   = \bigvee_{m \in \mathbb{N}} \{ 1 \wedge 0 \} = 0.
\]
So, $\neg \Diamond \neg \varphi(n) = 1$ for all $n \in \mathbb{N}$.

Note that by the way the relation $R$ is defined, these examples show that $\DIAMOND A \leftrightarrow \neg \BOX \neg A$ 
and  $\BOX A \leftrightarrow \neg \DIAMOND \neg A$ cannot be proved in {\sf Int2GC+FS}. This example also shows that
De~Morgan axioms~\eqref{Eq:DeMorganDual1} and \eqref{Eq:DeMorganDual2} cannot be proved in  {\sf G2GC+FS} neihter,
where {\sf G} stands for G{\"o}del--Dummett logic.
\end{example}
  
\begin{example} \label{Exa:Independent}
As we saw in Example~\ref{Exa:Motivation}(b), the first condition of \eqref{Eq:Dunn} does not imply the second one. Here we show that the
converse implication is not true neither. Thus, in Heyting algebras with two Galois connections, the
conditions of \eqref{Eq:Dunn} are independent.

Let $H = \{0, a, b, c, 1\}$ with $0 < c < a, b < 1$, but $a$ and $b$ are not comparable. We define $\Diamond, \Box \colon H \to H$ by
\begin{center}
$\Diamond 0 = \Diamond b = \Diamond c = 0
$ and $\Diamond a = \Diamond 1 = a
$; \\
$\Box 0 = \Box b = \Box c = b$
 and $\Box a = \Box 1 = 1$.
\end{center}
Then, we set $\DIAMOND = \Diamond$ and $\BOX = \Box$.
It is easy to verify that the pairs $(\Diamond, \BOX)$ and $(\DIAMOND,\Box)$ are Galois connections.

In addition, the second condition of \eqref{Eq:Dunn} holds for both $(\Diamond, \Box)$ and $(\DIAMOND, \BOX)$, that is,
\begin{center}
$\Box(x \vee y) \leq \Box x \vee \Diamond y
$ and $\BOX(x \vee y) \leq \BOX x \vee \DIAMOND y$.
\end{center}
But now the first conditions of \eqref{Eq:Dunn} does not hold, because, for instance, 
$\Diamond a \wedge \Box b = a \wedge b = c$, but $\Diamond(a \wedge b) = \Diamond c = 0$.
\end{example}

We end this section by noting that the above examples show that calculating using Heyting algebras with 
operators is much easier than calculating with categories, and calculating with algebras can be
easily used in showing some of the non-theorems, for instance.

\section{Representation theorem of H2GC+FS-algebras and relational completeness}
\label{Sec:RelationalSematics}

We introduced in \cite{DzJaKo10} relational frames and models for {\sf IntGC}. An {\sf IntGC}-frame
$(X,\leq,R)$ is a relational structure such that $(X,\leq)$ is a quasiordered set and $R$ is a relation on $X$ such that
\begin{equation} \label{Eq:frame1} 
({\ge} \circ R \circ {\ge}) \subseteq R .
\end{equation}
An {\sf IntGC}-model $(X,\leq,R,\models)$ is such that $(X,\leq,R)$ is an {\sf IntGC}-frame and
the satisfiability relation $\models$ is a binary relation from $X$ to the set of propositional
variables $\mathit{Var}$ such that $x \models p$ and $x \le y$ imply $y \models p$,
For any $x \in X$ and $A \in \Phi$, we define the satisfiability 
relation inductively by the following way:
\begin{align*}
x \models A \wedge B &\iff x \models A \mbox{ and } x \models B \\
x \models A \vee B &\iff x \models A \mbox{ or } x \models B \\
x \models A \to B &\iff \mbox{ for all } y \geq x, \  y \models A \mbox{ implies }  y \models B \\
x \models \neg A &\iff \mbox{ for no }y \geq x \mbox{ does }  y  \models A \\
x \models \Diamond A &\iff \mbox{ exists } y \mbox{ such that } x \, R \, y \mbox{ and }  y \models  A\\
x \models \BOX A &\iff \mbox{ for all } y, y \, R \, x \mbox{ implies }  y \models  A 
\end{align*}
We proved in \cite{DzJaKo10} that {\sf IntGC} is relationally complete, meaning that an {\sf IntGC}-formula $A$ is provable
if and only if $A$ is valid in all {\sf IntGC}-models, that is, for any  {\sf IntGC}-model
$(X,\leq,R,\models)$ and for all $x \in X$, we have $x \models A$.

It is clear that since {\sf Int2GC} is a fusion of two independent {\sf IntGC}s, relational frames
for {\sf Int2GC} are of the form $(X,\leq,R_1,R_2)$ such that $(X,\leq,R_1)$ and $(X,\leq,R_2)$
are {\sf IntGC}-frames. Relational completeness  can then be proved by defining complex frames for
{\sf Int2GC} and by applying the results of complex frames of {\sf IntGC} (cf.~\cite{Kurucz2007}).

Fischer Servi described relational frames and models for {\sf IK} in \cite{FishServ84}. We will apply
the same frames for {\sf Int2GC+FS}.
An {\sf IK}-frame is a triple $(X,\leq,R)$, where $(X,\leq)$ is a quasiordered set and $R$ is a relation on $X$ such that
\begin{equation}\label{Eq:frame2} 
(R \circ {\leq}) \subseteq ({\leq} \circ R) \text{ \quad and \quad } ({\geq} \circ R) \subseteq (R \circ {\geq}).
\end{equation}
Note that in \cite{OrlRew07,OrlRadRew13} these frames are called {\sf HK1}-frames. Because the second condition
of \eqref{Eq:frame2} is equivalent to $(R^{-1} \circ {\leq}) \subseteq ({\leq} \circ R^{-1})$, we have that
$(X,\leq,R)$ is an {\sf IK}-frame if and only if $(X,\leq,R^{-1})$ is an {\sf IK}-frame. The following
relationship holds between {\sf IntGC}-frames and  {\sf IK}-frames.

\begin{lemma}
A frame $(X,\leq,R)$ is an {\sf IK}-frame if and only if $(X,\leq,R \circ {\geq})$ and 
$(X,\leq, R^{-1} \circ {\geq})$ are {\sf IntGC}-frames.
\end{lemma}

\begin{proof}
Let $(X,\leq,R)$ be an {\sf IK}-frame. Then, 
\[
 ({\geq} \circ (R \circ {\geq}) \circ {\geq}) \subseteq  ({\geq} \circ R \circ {\geq})  \subseteq (R  \circ {\geq} \circ {\geq})
 \subseteq (R \circ {\geq}),
\]
Thus, \eqref{Eq:frame1} holds for $(R \circ {\geq})$. Similarly, since  $(X,\leq,R^{-1})$ is an {\sf IK}-frame,
\begin{align*}
({\geq} \circ (R^{-1} \circ {\geq}) \circ {\geq}) &\subseteq ({\geq} \circ (R^{-1} \circ {\geq})) =  (({\geq} \circ R^{-1}) \circ {\geq}) \\
& \subseteq ((R^{-1} \circ {\geq} ) \circ {\geq}) \subseteq (R^{-1} \circ {\geq} ).
\end{align*}
Hence, \eqref{Eq:frame1} holds for $(R^{-1} \circ {\geq})$ also.

Conversely, suppose  $(X,\leq,R \circ {\geq})$ and 
$(X,\leq, (R^{-1} \circ {\geq}))$ are {\sf IntGC}-frames. Then,
\begin{align*}
(R \circ {\leq}) = ({\geq} \circ R^{-1})^{-1} \subseteq ( {\geq} \circ (R^{-1} \circ {\geq}) \circ {\geq})^{-1} 
\subseteq  (R^{-1} \circ {\geq} )^{-1} = ({\leq}  \circ R).
\end{align*}
In addition,
\[
 ({\geq} \circ R) \subseteq (({\geq} \circ R) \circ {\geq} \circ {\geq}) = ({\geq} \circ (R \circ {\geq}) \circ {\geq}) 
\subseteq (R \circ {\geq}).
\]
Therefore, $R$ satisfies \eqref{Eq:frame2}.
\end{proof}

In {\sf IK}-models $(X,\leq,R,\models)$, the satisfiability relation $\models$ for $\vee$, $\wedge$, $\to$, 
and $\neg$ are defined as earlier, and satisfiability for  $\Diamond A$, and $\Box A$ are defined by:
\begin{align*}
x \models \Diamond A   &\iff \mbox{ exists } y \mbox{ such that } x \, (R \circ {\geq}) \, y \text{ and } y \models  A \\
x \models \Box A &\iff \mbox{ for all } y, x \, ({\leq} \circ R) \, y \text{ implies } y \models  A  
\end{align*}
For the remaining  $\DIAMOND A$ and $\BOX A$, we define the satisfiability relation by:
\begin{align*}
x \models \DIAMOND A   &\iff \mbox{ exists } y \mbox{ such that } y \, ({\leq} \circ R) \, x \text{ and } y \models  A \\
x \models \BOX A &\iff \mbox{ for all } y, y \, (R \circ {\geq}) \, x \text{ implies } y \models  A 
\end{align*}
In the sequel, we call these models  {\sf IK$^2$}-models. The idea is that the models are based on
{\sf IK}-frames, but satisfiability is defined twice: both for the pairs ($\Diamond, \Box$) and
($\DIAMOND, \BOX$).  A formula $A$ is \emph{relationally valid} if it is valid in every  {\sf IK$^2$}-model, that is,
$x \models A$ holds for all elements $x$ in the model. Note  that:
\begin{align*}
x \models \Diamond A   &\iff \mbox{ exists } y \mbox{ such that } x \, R \, y \text{ and } y \models  A \\
x \models \DIAMOND A   &\iff \mbox{ exists } y \mbox{ such that } y \, R \, x \text{ and } y \models  A 
\end{align*}
We may now give the following soundness result.

\begin{proposition} Every {\sf Int2GC+FS}-provable formula is relationally valid.
\end{proposition}

\begin{proof} We need to show that the axioms of {\sf Int2GC+FS} are valid in all {\sf IK$^2$}-models, and that the Galois 
connection rules preserve validity.

In \cite{FishServ84}, it is proved that axioms (FS1) are (FS3) are valid in all {\sf IK}-frames. As an example, 
we consider (FS2). Validity of (FS4) can be proved in a similar way. If (FS2) is not valid, 
then there exists $x \in X$ such that (i) $x \models \DIAMOND (A \to B)$, but (ii) $x \not \models \BOX A \to \DIAMOND B$.
By (i), there is $y \, R \, x$ such that (iii) $y \models A \to B$, and (ii) means that there is $z \geq x$
such that (iv) $z \models \BOX A$, but (v) $z \not \models \DIAMOND B$. We have $y \, (R \circ {\leq} ) \, z$, which
implies by \eqref{Eq:frame2} that $y \, ({\leq} \circ R) \, z$, meaning that there is $v \geq y$ such
that $v \, R \, z$. By (v), we get $v \not \models B$ and (iii) gives $v \not  \models A$. Now
$v \, (R \circ {\geq}) \, z$ implies $z \not \models \BOX A$, a contradiction to (iv).

Because the validity of $\Diamond$ and $\BOX$ are defined in terms of $R \circ {\geq}$ and its inverse, it
is clear that the pair $(\Diamond, \BOX)$ is a Galois connection on $\Phi$, that is, the rules
(GC\,${\Diamond}{\BOX})$ and (GC\,${\BOX}{\Diamond}$) preserve validity, and the same holds
for the pair ($\DIAMOND$, $\Box$).
\end{proof}

\begin{lemma} \label{Lem:Perisisent}
For all {\sf IK$^2$}-models $(X,\leq,R,\models)$ and formulas $A \in \Phi$: 
\[  x \models  A \mbox{ \ and \ } x \leq y \mbox{ \ imply \ }  y \models  A.\]
\end{lemma}

\begin{proof} We need to show the persistency of $\DIAMOND$ and $\BOX$, because other connectives
are considered in \cite{FishServ84}. Suppose
$x \models \DIAMOND A$ and $x \leq y$. Then, there is $z \, R \, x$ such that
$z \models A$. Now $z \, (R \circ {\leq}) \, y$ imply $z \, ({\leq} \circ R) \, y$.
Thus, $y \models \DIAMOND A$.

Assume that $x \models \BOX A$ and $x \leq y$. If $z \, (R \circ {\geq}) \, y$, then
also  $z \, (R \circ {\geq}) \, x$ and $z \models A$. So, $y \models \BOX A$.
\end{proof}

Or{\l}owska and Rewitzky studied canonical frames of HK1-algebras in \cite{OrlRew07}. For a 
HK1-algebra  $(H,\vee,\wedge,\to,0,1, \Diamond, \Box)$, let us denote for any $A \subseteq H$,
\[ {\Box}^{-1} A = \{ x \in H \mid \Box x \in A \} \text{ \quad and \quad }
   {\Diamond}^{-1}A = \{ x \in H \mid \Diamond x \in A\}.
\]
Let $X(H)$ be the set of all prime lattice filters of $H$. A relation $R^c$ is defined on $X(H)$ by
\begin{equation} \label{Eq:ComplexR}
 F \, R^c \, G \iff {\Box}^{-1} F \subseteq G \subseteq {\Diamond}^{-1} F .
\end{equation}
Or{\l}owska and Rewitzky showed that this frame is an {\sf IK}-frame. For an H2GC+FS-algebra
$(H,\vee,\wedge,\to,0,1, \Diamond, \Box,\DIAMOND, \BOX)$, its \emph{canonical frame} is
$(X(H), \subseteq, R^c)$. So, we are using the same canonical frames as for HK1-algebras. Let us denote: 
\[ {\BOX}^{-1} A = \{ x \in H \mid \BOX x \in A \} \text{ \quad and \quad }
   {\DIAMOND}^{-1}A = \{ x \in H \mid \DIAMOND x \in A\}.
\]
\begin{lemma} \label{Lem:CanFrameConnection}
Let  $(H,\vee,\wedge,\to,0,1, \Diamond, \Box,\DIAMOND, \BOX)$ be an H2GC+FS-algebra. Then, for all
$F,G \in X(H)$,
\begin{equation}
 F \, R^c G \iff  {\BOX}^{-1} G \subseteq F \subseteq {\DIAMOND}^{-1} G .
\end{equation}
\end{lemma}

\begin{proof}
Suppose $F \, R^c G$. If $x \in {\BOX}^{-1} G$, then $\BOX x \in G$. Now $\Diamond \BOX x \leq x$. Assume $x \notin F$. Then
$\Diamond \BOX x \notin F$, which  by the definition of $R^c$ gives  $\BOX x \notin G$, a contradiction.
Similarly, if $x \in F$, then $\Box \DIAMOND x \geq x$ and $\Box \DIAMOND x \in F$. This implies $\DIAMOND x \in G$
by the definition of $R^c$. 
Conversely, if ${\BOX}^{-1} G \subseteq F \subseteq \DIAMOND^{-1} G$, then $F \, R^c G$ can be proved in an analogous manner.
\end{proof}

Lemma~\ref{Lem:CanFrameConnection} means that we have two ways to define the relation $R^c$, either by using $\Diamond$ and $\Box$,
or by using $\DIAMOND$ and $\BOX$. Let $(X(H), \subseteq, R^c)$ be the canonical frame of some H2GC+FS-algebra on $H$. To obtain
the canonical model, we define the relation $\models_c$ from $X(H)$ to $\mathit{Var}$ by $F \models_c p$ if and only if $v(p) \in F$.

In the book \cite{OrlRadRew13}, it is shown that in the canonical frame $(X(H),\subseteq,R^c)$ of an HK1-algebra
$(H,\vee,\wedge,\to,0,1, \Diamond, \Box)$, for all $F,G \in X(H)$,
\begin{enumerate}[\rm (i)]
 \item $F \, ({\subseteq} \circ R^c) \, G \iff \Diamond^{-1}F \subseteq G$;
 \item $F \, (R^c \circ {\supseteq}) \, G \iff G \subseteq \Box^{-1} F$.
\end{enumerate}
We extend this result to H2GC+FS-algebras.

\begin{lemma} \label{Lem:CanFrameConnection2}
Let $(H,\vee,\wedge,\to,0,1, \Diamond, \Box,\DIAMOND, \BOX)$ be an H2GC+FS-algebra.
Then in the canonical frame $(X(H),\subseteq,R^c)$, for all $F,G \in X(H)$,
\begin{enumerate}[\rm (a)]
 \item $F \, ({\subseteq} \circ R^c) \, G \iff F \subseteq \DIAMOND^{-1}G$;
 \item $F \, (R^c \circ {\supseteq}) \, G \iff \BOX^{-1} G \subseteq F$.
\end{enumerate}
\end{lemma}

\begin{proof}
(a) Suppose that  $F \, ({\subseteq} \circ R^c) \, G$. Then, there is $J \in X(H)$ 
such that $F \subseteq J$ and $J \, R^c \, G$, that is, $\BOX^{-1}G \subseteq J \subseteq \DIAMOND^{-1}G$.
Then, $F \subseteq \DIAMOND^{-1}G$. For the other direction, assume that  $F \subseteq \DIAMOND^{-1}G$.
Consider the filter $K$ generated by $F \cup \BOX^{-1}G$. Suppose $K \cap - \DIAMOND^{-1}G \neq \emptyset$.
Then, there are $a \in - \DIAMOND^{-1}G$, $b \in F$ and $c \in \BOX^{-1}G$
such that $b \wedge c \leq a$. Note that $\BOX^{-1}G$ is a filter, so it is closed under finite meets. Hence,
$b \leq c \to a$, and $\DIAMOND b \leq \DIAMOND (c \to a) \leq (\BOX c \to \DIAMOND a)$ by
(FS3). Because $b \in F$ and $F \subseteq \DIAMOND^{-1}G$, we have $\DIAMOND b \in G$
and so $\BOX c \to \DIAMOND a \in G$. Now $\BOX c \in G$ implies $\DIAMOND a \in G$,
a contradiction. Thus,  $K \cap - \DIAMOND^{-1}G = \emptyset$. It is easy to see that
$- \DIAMOND^{-1}G$ is an ideal. By the Prime Filter Theorem of distributive lattices, there is $J \in X(H)$
such that $K \subseteq J$ and $J \cap - \DIAMOND^{-1}G = \emptyset$, that is,
$J \subseteq  \DIAMOND^{-1}G$. Now $F \subseteq K \subseteq J$. Also
$\BOX^{-1}G \subseteq K$ by the definition of $K$. So, $\BOX^{-1}G \subseteq J \subseteq \DIAMOND^{-1}G$,
that is, $J \, R^c \, G$ by Lemma~\ref{Lem:CanFrameConnection}. Thus, $F \, ({\subseteq} \circ R^c) \, G$.

(b) Assume that $F \, (R^c \circ {\supseteq}) \, G$. Then, there exists $J \in X(H)$ such that
$F \, R^c \, J$ and $G \subseteq J$. Hence, $\BOX ^{-1}J \subseteq F$ and $\BOX^{-1}G \subseteq \BOX^{-1}J$
imply $\BOX^{-1}G \subseteq F$. On the other hand, assume  $\BOX^{-1}G \subseteq F$. Let $K$
be the filter generated by $G \cup \DIAMOND F$, where $\DIAMOND F = \{ \DIAMOND x \mid x \in F\}$.
Since $F$ is a prime filter, its complement $-F$ is a prime ideal. So, $-F \neq \emptyset$
and $\BOX(-F) = \{ \BOX x \mid x \notin F\} \neq \emptyset$. Let $I$ be an ideal generated by $\BOX(-F)$. Assume for
contradiction that $K \cap I \neq \emptyset$. Then, there exist $a \in G$,
$b_1, \cdots , b_m \in F$, and $d \in I$ such that $a \wedge \DIAMOND b_1 \wedge \cdots \wedge \DIAMOND b_m \leq d$. 
Take $b = b_1 \wedge \cdots \wedge b_m \in F$. We note that 
$\DIAMOND b = \DIAMOND (b_1 \wedge \cdots \wedge b_m) \le \DIAMOND b_1 \wedge \cdots \wedge \DIAMOND b_m$. 
Since $d \in I$, there are $c_1,c_2,\ldots,c_n \notin F$ such that $d \leq \BOX c_1 \vee \BOX c_2 \vee \cdots \vee \BOX c_n
\leq \BOX (c_1 \vee c_2 \vee \cdots \vee c_n)$. Because $F$ is a prime filter,
$c = c_1 \vee c_2 \vee \cdots \vee c_n \notin F$. Since $a \wedge \DIAMOND b \leq d \leq \BOX c$,
we have $a \leq \DIAMOND b \to \BOX c \leq \BOX(b \to c)$ by (FS4). Now $a \in G$ implies $\BOX(b \to c) \in G$.
Because $\BOX^{-1}G \subseteq F$, we have $b \to c \in F$. But now $b \in F$ implies $c \in F$,
a contradiction. Hence, $K \cap I = \emptyset$. By the Prime Filter Theorem of distributive lattices, 
there is $J \in X(H)$ such that $K \subseteq J$ and $J \cap I = \emptyset$. By the definition of $K$,
we have $G, \DIAMOND F \subseteq K\subseteq J$. In addition, $J \subseteq -\BOX(-F)$.
Therefore, $\BOX^{-1}J \subseteq \BOX^{-1}(-\BOX(-F)) = - \BOX^{-1}(\BOX(-F))
\subseteq F$. Because $\DIAMOND F \subseteq J$, we obtain $F \subseteq \DIAMOND^{-1}J$.
Thus, $\BOX^{-1}J \subseteq F \subseteq \DIAMOND^{-1}J$ and $F \, R^c \, J$.
Since $G \subseteq J$, we have  $F \, (R^c \circ {\supseteq}) \, G$.
\end{proof}

\medskip

For an {\sf IK}-frame $(X,\leq,R)$, let $\mathcal{T}_\leq$ be the set of $\leq$-closed subsets of $X$, that is,
\begin{equation}
 \mathcal{T}_\leq = \{ A \subseteq X \mid (\forall x,y \in X) \, x \in A \ \& \ x \leq y \Rightarrow y \in A \}.
\end{equation}
Then, $\mathcal{T}_\leq$ is an \emph{Alexandrov topology}, that is, it is a topology closed also under arbitrary intersections.
Another common name used for an Alexandrov topology is a \emph{complete ring of sets}. Let us denote by $\mathcal{I}_\leq \colon \wp(X) \to \wp(X)$ 
the interior operator of the topology $\mathcal{T}_\leq$, that is, for all $A \subseteq X$,
\[
 \mathcal{I}_\leq (A) = \bigcup \{ B \in \mathcal{T}_\leq \mid B \subseteq A \}.
\]
This means that $\mathcal{T}_\leq = \{ \mathcal{I}_\leq (A) \mid A \subseteq X\}$. The lattice $(\mathcal{T}_\leq,\subseteq)$
forms a Heyting algebra such that for all $A,B \in  \mathcal{T}_\leq$,
\[ 
A \to^c B =  \mathcal{I}_\leq (- A \cup B). 
\]
Let us define for a relational frame $(X,\leq,R)$ the following four operators $\wp(X) \to \wp(X)$:
\begin{align*}
 {\Box}^c A &= \{ x \in X \mid x \, ({\leq} \circ R) \, y \Rightarrow y \in A \} \\
 {\Diamond} ^c A &= \{ x \in X \mid (\exists y) \, x \, R \, y \ \& \ y \in A \} \\
 {\BOX}^c A &= \{ x \in X \mid y \, (R \circ {\geq}) \, x \Rightarrow y \in A \} \\
 {\DIAMOND} ^c A &= \{ x \in X \mid (\exists y) \, y \, R \, x \ \& \ y \in A \}.
\end{align*}
Or{\l}owska and Rewitzky \cite{OrlRew07} proved that
\[
 (\mathcal{T}_\leq,\cup,\cap,\to^c,\emptyset,X,{\Diamond}^c,{\Box}^c)
\]
is an HK1-algebra. It is clear that since also $(X,\leq,R^{-1})$ is an {\sf IK}-frame,
and $\DIAMOND$ is defined in terms of the inverse $R^{-1}$ of $R$ and $\BOX$ is defined
in terms of the inverse of $(R \circ {\geq})$ and $(R \circ {\geq})^{-1} = ({\leq} \circ R^{-1})$, the algebra
\[
 (\mathcal{T}_\leq,\cup,\cap,\to^c,\emptyset,X,{\DIAMOND}^c,{\BOX}^c)
\]
is an HK1-algebra. Obviously, the pairs $(\Diamond,{\BOX})$
and $(\DIAMOND, {\Box})$ are Galois connections on $(\mathcal{T}_\leq,\subseteq)$. Note that:
\begin{align*}
 {\Diamond} ^c A &= \{ x \in X \mid (\exists y) \, x \,  (R \circ {\geq}) \, y \ \& \ y \in A \} \\
 {\DIAMOND} ^c A &= \{ x \in X \mid (\exists y) \, y \,  ({\leq} \circ R) \, x \ \& \ y \in A \}.
\end{align*}
This then implies that 
\[
 C(X) = (\mathcal{T}_\leq,\cup,\cap,\to^c,\emptyset,X,{\Diamond}^c, {\Box}^c, {\DIAMOND}^c,{\BOX}^c)
\]
is an H2GC+FS-algebra, called the \emph{complex H2GC+FS-algebra} of the {\sf IK}-frame $(X,\leq,R)$.

Let  $(H,\vee,\wedge,\to,0,1, \Diamond, \Box,\DIAMOND, \BOX)$ be an H2GC+FS-algebra. We define a 
mapping $h \colon H \to C(X(H))$ from $H$ to the complex algebra of its canonical frame by
\[
 h(x) = \{ F \in X(H) \mid x \in F \}.
\]
It is proved in \cite{OrlRadRew13} that 
\begin{align*}
 h(x \vee y) &= h(x) \cup h(y) & h(x \wedge y) &=  h(x) \cap h(y) \\
 h(x \to y ) &= h(x) \to ^c h(y) & \\
 h(0) &= \emptyset & h(1) &= X(H) \\
 h(\Diamond x ) &= {\Diamond}^c h(x) & h(\Box x) &= {\Box}^c h(x)
\end{align*}
Note that the proof of Lemma 4.4 in \cite{OrlRew07} contains some mistakes, but the proof is
corrected in \cite{OrlRadRew13}. We can extend this result to H2GC+FS-algebras.

\begin{lemma} \label{Lem:Homomorphism}
Let  $(H,\vee,\wedge,\to,0,1, \Diamond, \Box,\DIAMOND, \BOX)$ be an H2GC+FS-algebra.
\begin{enumerate}[\rm (a)]
 \item $h(\DIAMOND x) =  {\DIAMOND}^c h(x)$;
 \item $h(\BOX x) = {\BOX}^c h(x)$.
\end{enumerate}
\end{lemma}

\begin{proof}
(a) Suppose that $F \in {\DIAMOND}^c h(x)$. This means that there is $G \in X(H)$ such that $G \, R^c \, F$
and $G \in h(x)$. Now, $x \in G \subseteq \DIAMOND^{-1}F$ and $\DIAMOND x \in F$, that is, $F \in h(\DIAMOND x)$.
On the other hand, assume $F \in h(\DIAMOND x)$, that is, $\DIAMOND x \in F$. Suppose
${\uparrow}x \cap - \DIAMOND^{-1}F \neq \emptyset$. Then, there exists $y \in H $ such that
$x \leq y$ and $y \notin \DIAMOND^{-1}F$. We have $\DIAMOND y \notin F$ and $\DIAMOND x \leq \DIAMOND y$,
which give $\DIAMOND x \notin F$, a contradiction. Therefore, ${\uparrow}x \cap - \DIAMOND^{-1}F = \emptyset$
and ${\uparrow}x \subseteq \DIAMOND^{-1}F$. Now ${\uparrow}x$ is a filter and $-\DIAMOND^{-1}F$
is an ideal, as we already noted. Then, by the Prime Filter Theorem of distributive lattices, there is a prime filter 
$K \in X(H)$ such that ${\uparrow}x \subseteq K$ and $K \cap -\DIAMOND^{-1}F = \emptyset$, that is, $K \subseteq \DIAMOND^{-1}F$. 
Lemma~\ref{Lem:CanFrameConnection2}(a) implies $K \, ({\subseteq} \circ R^c) \, F$. Since $x \in K$, we have $K \in h(x)$ and so 
$F \in {\DIAMOND}^c h(x)$.

(b) Suppose $F \in h(\BOX x)$, that is, $\BOX x \in F$. If $G \, (R^c \circ {\supseteq}) \, F$, then there is
$K \supseteq F$ such that $\BOX^{-1}K \subseteq G \subseteq \DIAMOND^{-1}K$. Then, $\BOX x \in F \subseteq K$
and $x \in \BOX^{-1} K \subseteq G$. We get $G \in h(x)$ and $F \in {\BOX}^c h(x)$, as required.
Conversely, assume that $F \notin h(\BOX x)$. Then, $x \notin \BOX^{-1}F$ and ${\downarrow}x \cap \BOX^{-1}F = \emptyset$.
Note that $\BOX^{-1}F$ is a filter and ${\downarrow}x$ is an ideal. This means that by the Prime Filter Theorem of distributive lattices, there is a prime filter
$G \in X(H)$ such that $\BOX^{-1}F \subseteq G$ and ${\downarrow} x \cap G = \emptyset$. Then,
$x \notin G$, $G \notin h(x)$, and by Lemma~\ref{Lem:CanFrameConnection2}(b), $G \, (R^c \circ {\supseteq}) \, F$.
Thus, $F \notin {\BOX}^c h(x)$.
\end{proof}

As a corollary, we write the following representation theorem for H2GC+FS-algebras.

\begin{theorem} \label{Thm:Representation}
Every H2GC+FS-algebra can be embedded into the complex algebra of its canonical {\sf IK}-frame.
\end{theorem}

\begin{remark} \label{Rem:Representation}
In general, the embedding $h$ is not an isomorphism, but in some cases it can be also surjective. For instance,
in \cite[Theorem~7.2]{DzJaKo10}, we showed that every finite HGC-algebra  is isomorphic to the complex
algebra of its canonical frame, and a similar proof could be presented here. More generally, in \cite[Theorem~18]{DzJaKo14},
we showed that for every spatial HGC-algebra $\mathbb{H}$, there exists an {\sf IntGC}-frame $\mathcal{F}$ such that 
$\mathbb{H}$ is isomorphic to the complex algebra of $\mathcal{F}$. The same idea could be applied for H2GC+FS-algebras,
because it is known that spatial (and thus complete) Heyting algebras are order-isomorphic to some Alexandrov topologies.

Note also that a representation theorem for tense symmetric Heyting algebras is given in \cite{figallo2012TSH}, but their
algebras differ essentially from ours. 
\end{remark}

In terms of Theorem~\ref{Thm:Representation}, we can prove the Key Lemma.

\begin{lemma}\label{Lem:Key}
Let  $(H,\vee,\wedge,\to,0,1, \Diamond, \Box,\DIAMOND, \BOX)$ be an H2GC+FS-algebra.
Then, for all $A \in \Phi$ and $F \in X(H)$,
\begin{equation} \label{Eq:Key}
 F \models_c A \iff v(A) \in F. 
\end{equation}
\end{lemma}

\begin{proof} We consider the operators $\Diamond$ and $\BOX$ only, because for connectives $\vee$, $\wedge$, $\to$
the claim is well known, and for $\DIAMOND$ and $\Box$ the proof is analogous.

($\Diamond$) \ Suppose $A = \Diamond B$ for some $B \in \Phi$ and $B$ satisfies \eqref{Eq:Key}. If $F \models_c \Diamond B$, then
there is a prime filter $G$ such that $F \, R^c \, G$ and $G \models_c B$, that is, $v(B) \in G$. Now,
$G \subseteq \Diamond^{-1}F$ implies $v(B) \in \Diamond^{-1}F$ and $v(A) = v(\Diamond B) = \Diamond v(B) \in F$.
Conversely, suppose $v(A) \in F$, that is, $F \in h(v(A)) = h(\Diamond v(B)) = {\Diamond}^c h(v(B))$. Then, there
exists $G \in X(H)$ such that $G \in h(v(B))$ and $F \, R^c \, G$, that is, $v(B) \in G$, or equivalently $G \models_c B$. Hence,
$F \models_c \Diamond B$.

($\BOX$) \ Suppose that $A = \BOX B$ and $B$ satisfies \eqref{Eq:Key}. Assume that $v(\BOX B) = \BOX v(B) \in F$. If
$G \, (R^c \circ {\supseteq}) \, F$, then by  Lemma~\ref{Lem:CanFrameConnection2}(b),
$\BOX^{-1} F \subseteq G$. Then, $\BOX v(B) \in F$ gives $v(B) \in G$, $G \models_c B$, and $F \models_c \BOX B$.
On the other hand, if $F \models_c \BOX B$, then $G \, (R^c \circ {\supseteq}) \, F$ implies $G \models_c B$, that is,
$v(B) \in G$ and $G \in h(v(B))$. Hence, $F \in {\BOX}^c h(v(B)) = h ( \BOX v(B)) = h(v(\BOX B))$ and $v(\BOX B) \in F$. 
\end{proof}

We are now able to prove relational completeness.

\begin{theorem}
Every formula $A \in \Phi$ is {\sf Int2GC+FS}-provable if and only if $A$ is relationally valid.
\end{theorem}

\begin{proof} We have already noted that every {\sf Int2GC+FS}-provable formula is relationally valid.
On the other hand, if $A$ is not  {\sf Int2GC+FS}-provable, there exists an H2GC+FS-algebra on some set $H$
and a valuation $v$ such that $v(A) \ne 1$. Let $(X(H),\subseteq,R,\models_c)$ be the corresponding canonical frame. 
Now,  $h(v(A)) \neq X(H)$, which implies that there is a prime filter $F$ such that $v(A) \notin F$.
Using the Key Lemma, this implies $F \not \models_c A$, and thus $A$ is not relationally valid.
\end{proof}

Kripke completeness of {\sf K$_t$} is provided by Kripke frames $(X,R)$, where $R$ is an arbitrary binary relation on $X$. Thus, relational
completeness for {\sf Cl2GC+FS} is standard. Note also that the frames $(X,R)$ can be considered as relational
{\sf IK}-frames $(X,R,=)$ and then the satisfiability relation is the same in both settings.
The result analogous to Theorem~\ref{Thm:Representation} but for classical logic follows from the fundamental representation theorem for 
Boolean algebras with operators by J{\'o}nsson and Tarski \cite{JoTa51}, which says that every Boolean algebra with additive and normal 
operators can be embedded into the complex of its canonical frame. Therefore, if we have a Boolean algebra with two Galois connections $(\Diamond, \BOX)$
and $(\DIAMOND, \Box)$, then the operators $\Diamond$ and $\DIAMOND$ are additive and normal. Since $\Box$ and $\BOX$ are connected to
$\Diamond$ and $\DIAMOND$ by Fischer Servi axioms, this means that they are completely determined by De~Morgan dualities.

As far as relational semantics is concerned, moving from intuitionistic logic (and classical logic) to intermediate logics, as a ``base logic'', 
is no longer as straightforward as in the case of algebraic semantics. Some intermediate logics are Kripke-incomplete, that is, they do not have adequate 
relational semantics. Since canonicity proofs for intermediate logics have different patterns (there are various kinds of canonicity 
like hypercanonicity, $\omega$-canonicity, extensive canonicity; see \cite{GhiMig99}, for instance), a uniform approach to completeness for 
intermediate logics seems unlikely.  Therefore, the relational completeness and the representation theorems in case of particular intermediate 
logic are left for separate studies. 

\section{Some concluding remarks}

We have introduced method (A), which for each intermediate logic {\sf L} uniformly defines the corresponding tense logic {\sf LK$_t$}.
Method (B) introduced by Davoren \cite{Davoren} cannot be applied to every intermediate logic {\sf L}, because this method first
builds the fusion {\sf LK}\FUSE{\sf LK} of two copies of intuitionistic modal logic {\sf LK} and then adds Brouwerian axioms interlinking
the modalities -- but in many cases, it is unclear, what  the modal logic {\sf LK} actually is.

Approach (A) allows a uniform treatment of algebraic semantics and we have shown algebraic completeness of {\sf L2GC+FS} for any intermediate logic {\sf L}. 
It is well known that there are many intermediate logics that are proved to be Kripke-incomplete and also there are several intermediate
logics for which even their frames are not known.
Here, we have presented a completeness theorem for {\sf Int2GC+FS} = {\sf IK$_t$} in a way that uses 
{\sf IK}-frames introduced by Fischer Servi \cite{FishServ84}. This hints that in an analogous way, Kripke completeness of {\sf L2GC+FS} = {\sf LK$_t$} 
can be proved for  several intermediate modal logics {\sf L} that are known to be at least Kripke-complete. For instance, G{\"o}del--Dummett 
logic is characterized by the frames $(X,\leq)$ such that $(x \leq y \mbox { and } x \leq z) \: \Rightarrow \: (y  \leq z \mbox { or  } z \leq y)$. 
For other examples, see \cite{ChaZak97}. Notice also that different intermediate logics may need different tools to carry out a completeness proof, and these
``apparatuses'' are considered in \cite{GhiMig99}.

We have also presented a representation theorem stating that every H2GC+FS-algebra can be embedded into the complex 
algebra of its canonical {\sf IK}-frame. A similar proof for some intermediate logic algebras probably can be obtained, but this requires careful study of
complex algebras and canonical frames. These will be studied in the future.

\section*{Acknowledgement}

The authors of this paper thank the anonymous referee for the comments and suggestions which helped us to improve our paper significantly.


\end{document}